\documentclass[a4paper,12pt]{amsart}
\usepackage{environ}
\usepackage{fancyhdr}
\usepackage{mabliautoref}
\usepackage{amssymb,amsthm,amsmath}
\RequirePackage[dvipsnames,usenames]{xcolor}
\usepackage{hyperref}
\usepackage{cleveref}
\usepackage{mathtools}
\usepackage{tcolorbox}
\usepackage[all]{xy}
\usepackage{tikz}
\usepackage{tikz-cd}
\usetikzlibrary{decorations.markings,shapes,positioning}
\usepackage{enumitem}
\usepackage{stmaryrd}
\usepackage{xfrac}
\usepackage{fix-cm}
\usepackage{faktor}
\usepackage{minibox}
\usepackage{commath}
\usepackage{chngcntr}
\usepackage{setspace}
\usepackage{array}
\usepackage[toc,page]{appendix}
\usepackage{calligra,mathrsfs}
\usepackage{mathtools}
\usepackage{bbm}
\hypersetup{
	bookmarks,
	bookmarksdepth=3,
	bookmarksopen,
	bookmarksnumbered,
	pdfstartview=FitH,
	colorlinks,backref,hyperindex,
	linkcolor=Sepia,
	anchorcolor=BurntOrange,
	citecolor=Cyan,
	filecolor=BlueViolet,
	menucolor=Yellow,
	urlcolor=OliveGreen
}

\newtcolorbox{activitybox}[1][]{%
	breakable,
	enhanced,
	colback=lightergray,
	boxrule=3pt,
	arc=5pt,
	outer arc=5pt,
	boxsep=10pt,
	colframe=darkergray,
	coltitle=white,
	#1
}

\newcommand{\quot}[2]{%
	\raise1ex\hbox{$#1$}\Big/\lower1ex\hbox{$#2$}%
}

\newcommand{\colim}{\varinjlim}
\renewcommand{\lim}{\varprojlim}

\makeatletter\def\Rvarlim@#1#2{%
	\vtop{\m@th\ialign{##\cr
			\hfil$#1\operator@font Rlim$\hfil\cr
			\noalign{\nointerlineskip\kern1.5\ex@}#2\cr
			\noalign{\nointerlineskip\kern-\ex@}\cr}}%
}
\makeatother

\makeatletter\def\Rlim{%
	\mathop{\mathpalette\Rvarlim@{\leftarrowfill@\textstyle}}\nmlimits@
}
\makeatother

\newcommand{\expl}[2]{\underset{\mathclap{\minibox[c]{$\uparrow$\\ \fbox{\footnotesize #2}}}}{#1}}

\newcommand{\dra}{\dashrightarrow}

\def\xto#1#2{\xrightarrow{#1}{#2}}
\def\xto#1{\xrightarrow{#1}}
\newcommand{\surj}{\twoheadrightarrow}
\newcommand{\inj}{\hookrightarrow}
\newcommand{\ttilde}{\widetilde}

\newcommand{\HHom}{\mathcal{H}om}

\newcommand{\et}{\acute{e}t}

\newcommand{\stacksproj}[1]{{\cite[Tag~\href{http://stacks.math.columbia.edu/tag/#1}{#1}]{Stacks_Project}}}

\newcommand{\stacksprojs}[2]{\cite[Tags ~\href{http://stacks.math.columbia.edu/tag/#1}{#1} and ~\href{http://stacks.math.columbia.edu/tag/#2}{#2}]{Stacks_Project}}

\newcommand{\inc}{\subseteq}

\newcommand{\cni}{\supseteq}

\newcommand{\esp}{\mbox{ }}
\newcommand{\bighat}{\widehat}

\DeclareMathAlphabet{\mathchanc}{OT1}{pzc}%
{m}{it}

\renewcommand{\set}[2]{\{ \ #1 \  | \ #2 \ \}}

\newcommand{\bA}{\mathbb{A}}

\newcommand{\bF}{\mathbb{F}}

\newcommand{\bP}{\mathbb{P}}
\newcommand{\bQ}{\mathbb{Q}}

\newcommand{\bZ}{\mathbb{Z}}

\newcommand{\scr}{\mathcal}
\newcommand{\cA}{\scr{A}}

\newcommand{\cC}{\scr{C}}

\newcommand{\cF}{\scr{F}}

\newcommand{\cH}{\scr{H}}

\newcommand{\cK}{\scr{K}}

\newcommand{\cM}{\scr{M}}
\newcommand{\cN}{\scr{N}}
\newcommand{\cO}{\scr{O}}
\newcommand{\cP}{\scr{P}}

\newcommand{\cR}{\scr{R}}

\newcommand{\cV}{\scr{V}}

\newcommand{\cX}{\scr{X}}

\DeclareMathOperator{\injj}{{inj}}

\DeclareMathOperator{\FM}{{FM}}

\DeclareMathOperator{\alb}{{alb}}
\DeclareMathOperator{\Alb}{{Alb}}

\DeclareMathOperator{\Tr}{Tr}
\DeclareMathOperator{\codim}{codim}

\DeclareMathOperator{\Tor}{Tor}

\DeclareMathOperator{\id}{{id}}

\DeclareMathOperator{\im}{{im}}

\DeclareMathOperator{\length}{{length}}

\DeclareMathOperator{\Pic}{Pic}

\DeclareMathOperator{\rank}{{rank}}

\DeclareMathOperator{\red}{red}

\DeclareMathOperator{\reg}{reg}

\DeclareMathOperator{\Spec}{{Spec}}

\DeclareMathOperator{\Supp}{{Supp}}

\DeclareMathOperator{\Crys}{Crys}
\DeclareMathOperator{\coh}{coh}

\DeclareMathOperator{\Sh}{Sh}
\DeclareMathOperator{\QCoh}{QCoh}

\DeclareMathOperator{\crys}{crys}

\DeclareMathOperator{\Coh}{Coh}

\DeclareMathOperator{\ShHom}{\mathscr{H}\text{\kern -3pt {\calligra\large om}}\,}
\DeclareMathOperator{\sep}{sep}

\DeclareFontFamily{OT1}{pzc}{}
\DeclareFontShape{OT1}{pzc}{m}{it}{<-> s * [1.200] pzcmi7t}{}
\DeclareMathAlphabet{\mathpzc}{OT1}{pzc}{m}{it}

\renewcommand{\phi}{\varphi}

\newcommand{\factor}[2]{\left. \raise 2pt\hbox{\ensuremath{#1}} \right/
	\hskip -2pt\raise -2pt\hbox{\ensuremath{#2}}}

\makeatletter
\renewcommand\subsection{
	\renewcommand{\sfdefault}{pag}
	\@startsection{subsection}%
	{2}{0pt}{.8\baselineskip}{.4\baselineskip}{\raggedright
		\sffamily\itshape\small\bfseries
}}
\renewcommand\section{
	\renewcommand{\sfdefault}{phv}
	\@startsection{section} %
	{1}{0pt}{\baselineskip}{.8\baselineskip}{\centering
		\sffamily
		\scshape
		\bfseries
}}
\makeatother

\definecolor{gr}{rgb}{0,0.5,0}

\newcommand{\Addresses}{{% additional braces for segregating \footnotesize
		\bigskip
		\footnotesize
		
		\textsc{\'Ecole Polytechnique F\'ed\'erale de Lausanne, SB MATH CAG, MA C3 615 (B\^atiment MA), Station 8, CH-1015 Lausanne, Switzerland}\par\nopagebreak
		\textit{E-mail address}: \texttt{jefferson.baudin@epfl.ch}

}}

\usepackage[margin=1.0in]{geometry}
\setlist{  
	listparindent=\parindent,
	parsep=0pt,
}

\author[J.~Baudin]{Jefferson Baudin} 

\title[Effective characterization of ordinary abelian varieties]{Effective characterization of ordinary abelian varieties, and beyond}
\begin{document}
	\begin{abstract}
		We prove that the Albanese morphism of any normal proper variety $X$ in positive characteristic satisfying $S^0(X, \omega_X) \neq 0$ and $P_2(X) = 1$ is surjective with connected fibers, and that $\Alb(X)$ is ordinary.
		
		We obtain from a variant of the above a purely positive characteristic proof of Chen and Hacon's effective birational characterization of complex abelian varieties (\cite{Chen_Hacon_Characterization_of_abelian_varieties}).
	\end{abstract}

	\maketitle
	
	\tableofcontents
\section{Introduction}

%A common problem in birational geometry is to characterize certain classes of projective varieties by fixing birational invariants. Classicial surface results in this direction were Castelnuovo's criterion of rationality, and Enriques' characterization of abelian surfaces (see \JB{refs} for more recent results). Much more recently, several authors generalized Enriques' result to arbitrary dimension (\JB{ref Ein-Lazarsfeld, Kollar, Chen-Hacon, Zhi, Pareschi}). 

Chen and Hacon obtained in \cite{Chen_Hacon_Characterization_of_abelian_varieties} the following effective (and optimal) birational characterization of complex abelian varieties, in terms of the first Betti number and the first two plurigenera:

\begin{thm_blank*}
	A smooth complex projective variety if birational to an abelian variety if and only if $P_1(X) = P_2(X) = 1$ and $b_1(X) = 2\dim(X)$.
\end{thm_blank*}

%\begin{rem_blank}
%	Note that it follows from standard generic vanishing theory that this the second condition is equivalent to the fact that $P_1(X) = P_2(X) = 1$ and $b_1(X) = 2\dim(X)$.
%\end{rem_blank}

%The precise result we want to focus on here was obtained by Chen and Hacon, generalizing the aformentioned theorem of Enriques:
%
%
%A theorem of Chen and Hacon states that if $X$ is a smooth projective variety if maximal Albanese dimension such that $P_2(X) = 1$, then $X$ is birational to an abelian variety (\cite{Chen_Hacon_Characterization_of_abelian_varieties}). This was later generalized by \JB{ref Zhi and Pareschi, right?} 

Our objective is to prove a positive characteristic analogue of the above. This was achieved for surfaces in \cite{Ferrari_An_Enriques_classification_theorem_for_surfaces_in_positive_characteristic}, and a non--effective version was proved in arbitrary dimension in \cite{Hacon_Patakfalvi_Zhang_Bir_char_of_AVs}. Let us state our main result:

\begin{thm_letter}[{\autoref{main_thm_wo_ss_factor}, \autoref{main_thm_Zhang_technique}}]\label{intro:thm_A}
	A normal proper variety $X$ in characteristic $p > 0$ is birational to an ordinary abelian variety if and only if the following holds:
	\begin{itemize}
		\item $S^0(X, \omega_X) \neq 0$;
		\item $P_2(X) = 1$;
		\item $b_1(X) = 2\dim(X)$;
		\item $\alb_X$ is separable.
	\end{itemize}
	Furthermore, the fourth condition is automatic whenever $P_4(X) = 1$ or $P_p(X) \leq p - 1$.
\end{thm_letter}

\begin{rem_blank}
	Somewhat funnily, this shows that the optimal result when $p = 2$. We believe that separability of the Albanese morphism should follow from the condition $P_2(X) = 1$ for any value of $p > 0$.
\end{rem_blank}

As it turns out, we can relax a bit the ordinarity hypothesis above and prove the following:

\begin{thm_letter}[{\autoref{main_thm_wo_ss_factor}}]\label{intro:thm_B}
	Let $X$ be a normal proper variety of maximal Albanese dimension. Assume that $P_2(X) = 1$ and that $\Alb(X)$ has no supersingular factor. Then $\alb_X \colon X \to \Alb(X)$ is surjective, and the finite part of its Stein factorization is purely inseparable. 
	
	If furthermore $P_p(X) \leq p - 1$, then $X$ is birational to an abelian variety.
\end{thm_letter}

Thanks to \autoref{intro:thm_B} and a weak version of the ordinarity conjecture due to Pink (\cite{Pink_On_the_order_of_the_reduction_of_a_point_on_an_abelian_variety}), we obtain a purely positive characteristic proof of the following rephrasement of Chen and Hacon's theorem.

\begin{thm_letter}[{\autoref{Chen_Hacon_characterization}}]\label{intro_Chen_Hacon}
	Let $X$ be a normal proper complex variety of maximal Albanese dimension such that $P_2(X) = 1$. Then $X$ is birational to an abelian variety.
\end{thm_letter}

We can also generalize \autoref{intro:thm_A} for varieties which do not necessarily have maximal Albanese dimension. 

\begin{thm_letter}[{\autoref{main_thm_Delta_integral}}]\label{intro:general_case}
	Let $X$ be a normal proper variety and let $D \geq 0$ be a divisor on $X$. Assume that $S^0(X, D; \cO_X(K_X + D)) \neq 0$ and that $P_2(X, D) = 1$. Then $\alb_X$ is surjective with connected fibers and $\Alb(X)$ is ordinary.
	
	If furthermore $P_p(X, D) \leq p - 1$, then $\alb_X$ is a separable fibration.
\end{thm_letter}

\begin{rem_blank}
	\begin{itemize}
		\item By $P_r(X, D)$, we mean the quantity $h^0(X, \cO_X(r(K_X + D)))$.
		\item  When $D = 0$, the characteristic zero analogue of \autoref{intro:general_case} is present in \cite[Theorem 3.1]{Jiang_An_effective_version_of_a_thm_of_Kawamata_on_the_Albanese_map}, as a consequence of \autoref{intro_Chen_Hacon}, some Hodge theory and Kollár's decomposition theorem \cite{Kollar_Higher_Direct_Image_of_Dualizing_Sheaves_II}. We are not aware of any characteristic zero analogue when $D \neq 0$.
	\end{itemize} 
\end{rem_blank}

As a consequence of the proof of \autoref{intro:general_case}, we obtain the following version of Kawamata's theorem (\cite{Kawamata_Characterization_of_abelian_varieties}):

\begin{thm_letter}[{\autoref{main_thm_kod_dim_zero}, \autoref{main_thm_Hacon_Pat_Zhang}}]\label{intro:thm_E}
	Let $X$ be a normal proper variety, let $D \geq 0$ be a $\bZ_{(p)}$--divisor and assume that $\kappa_S(K_X + D) = 0$. Then $\alb_X \colon X \to \Alb(X)$ is a separable fibration and $\Alb(X)$ is ordinary.
\end{thm_letter}

\begin{rem_blank}
	We refer the reader to \autoref{def:kod_stable_kod_dim} for the definition of the Frobenius stable Kodaira dimension $\kappa_S$, and to \autoref{main_thm_Hacon_Pat_Zhang} for a slight extension of \autoref{intro:thm_E}, generalizing the main result of \cite{Hacon_Patakfalvi_Zhang_Bir_char_of_AVs}.
\end{rem_blank}

\subsection{Ideas of the proof}

The fundamental technique to prove Chen and Hacon's theorem is \emph{generic vanishing theory} (\cite{Green_Lazarsfeld_Generic_vanishing, Green_Lazarsfeld_GV_2}), which deals with cohomological properties of the sheaf $\omega_X$ (and its pushforward to $\Alb(X)$) and deduces geometric consequences from it. Although the direct positive characteristic retranscription of the generic vanishing theorem is known to fail (\cite{Filipazzi_GV_fails_in_pos_char, Hacon_Kovacs_GV_fails_in_pos_char}), the breakthrough of Hacon and Patakfalvi in \cite{Hacon_Pat_GV_Characterization_Ordinary_AV} was to remember that the sheaf $\omega_X$ carries the \emph{Cartier operator} (which is the Grothendieck dual of the $p$'th power map $\cO_X \to \cO_X$). They managed to prove that $\omega_X$ (or more precisely $\alb_{X, *}(\omega_X)$) satisfies generic vanishing \emph{up to nilpotence under the Cartier operator}. We refer the reader to \cite{Baudin_Positive_characteristic_generic_vanishing_theory} for more information on the precise meaning of this, and more generally for an introduction to positive characteristic generic vanishing theory. 

Let us now explain ideas from the first part of the proof of \autoref{intro:thm_B}. The approach we follow is mostly that of \cite{Pareschi_Basic_results_on_irr_vars_via_FM_methods}, so the main actor of the proof if the closed subset $V^0_{\injj}(\omega_X) \inc \Pic^0(X)$ (a version ``up to nilpotence'' of the usual cohomology support locus $V^0(\omega_X) = \set{\alpha \in \Pic^0(X)}{H^0(X, \omega_X \otimes \alpha) \neq 0}$). The proof decomposes in two steps: 
\begin{enumerate}
	\item \emph{Show that $V^0_{\injj}(\omega_X)$ is finite.} Here, we really follow the proof of \cite[Theorem 5.1]{Pareschi_Basic_results_on_irr_vars_via_FM_methods} (the key part being Lemma 4.2 from \emph{loc. cit.}). An immediate remark to make is that Pareschi's proof uses several fundamental theorems which are known to fail in positive characteristic (Grauert--Riemenschneider vanishing as well as Kollár's decomposition and torsion freeness theorems), even up to nilpotence (see \cite{Baudin_Bernasconi_Kawakami_Frobenius_GR_fails}). Fortunately, it turns out that in our very specific case, we can prove ``by hand'' that everything we need works. This is mostly made possible thanks to our recent ``categorification'' of generic vanishing (\cite[Theorem 3.2.1]{Baudin_Positive_characteristic_generic_vanishing_theory}).
	\item \emph{Deduce that the finite part of $\alb_X$ is purely inseparable.} Here, we diverge from the methods of \cite{Pareschi_Basic_results_on_irr_vars_via_FM_methods}. As in characteristic zero, we know at this point that $\alb_{X, *}(\omega_X)$ is a unipotent vector bundle (up to nilpotence), and the goal is to show that its rank is one (up to nilpotence). A characteristic zero approach is to use the trace map so deduce that any non--zero map $\alb_{X, *}(\omega_X) \to \cO_{\Alb(X)}$ splits (one can also use metrics as in \cite{Hacon_Popa_Schnell_Alg_fiber_spaces_over_abelian_varieties}), so $\alb_{X, *}(\omega_X)$ is trivial. Given that $P_1(X) = 1$, this forces $\alb_{X, *}(\omega_X) \cong \cO_{\Alb(X)}$, so $\alb_X$ is birational since it has rank one.
	
	Although trace maps do not behave well in positive characteristic, we can split maps as above via the use of a base change by a certain isogeny; the Verschiebung endomorphism. This is the approach taken in \cite{Hacon_Patakfalvi_Zhang_Bir_char_of_AVs}, and we replicate it here (up to certain technicalities). However, since we do an étale base change, we cannot bound $P_1(X)$ anymore. The way around is not to look at $P_1(X) = h^0(X, \omega_X)$, but at $h^{\dim(X)}(X, \omega_X)$, which is \emph{always} $1$. 
\end{enumerate}

\subsection{Acknowledgments}
I would like to thank Marta Benozzo, Fabio Bernasconi, Iacopo Brivio, Stefano Filipazzi, Lucas Gerth, János Kollár, Léo Navarro Chafloque, Zsolt Patakfalvi, Mihnea Popa, Sofia Tirabassi and Jakub Witaszek for discussions related to the content of this article. Financial support was provided by the grants \#200020B/192035 and \#2000201-231484 from the Swiss National Science Foundation, as well as the grant \#804334 from the European Research Council.

\subsection{Notations}
\begin{itemize}
	\item We fix a prime number $p > 0$. Unless stated otherwise, all rings and schemes in this paper are defined over $\bF_p$. Their absolute Frobenius is denoted by $F$. 
	\item A \emph{variety} is a separated and integral scheme of finite type over a field. 
	\item The symbol $A$ always denotes an abelian variety of dimension $g$ and $\bighat{A}$ denotes its dual.
	\item An \emph{isogeny} is a finite and surjective morphism of abelian varieties.
	\item A \emph{fibration} is a proper morphism $f \colon X \to Y$ between Noetherian schemes, such that $f_*\cO_X = \cO_Y$.
	\item Given a scheme $X$, its reduction is denoted $X_{\red}$. If $X$ is integral, then $K(X)$ denotes its fraction field.
	\item If $x \in X$ is any point, we write $i_x \colon \Spec k(x) \to X$ for the natural map.
	%\item If $x \in X$ is a closed point and $G$ is an abelian group, then $G(x)$ denotes the skyscraper sheaf at $x$ with value $G$.
	\item If $X$ is a normal proper variety, $D \geq 0$ is a divisor and $r \in \bZ$, we let $P_r(X, D) \coloneqq h^0(X, \cO_X(r(K_X + D))$.
	\item Given $f\colon X \to Y$ a morphism of schemes and let $y \in Y$. The fiber of $f$ at $y$ will be denoted $X_y$. If $L$ (resp. $D$) is a line bundle (resp. a Cartier divisor) on $X$, then we write $L_y \coloneqq L|_{X_y}$ (resp. $D_y \coloneqq D|_{X_y}$).
	\item Given a generically finite morphism $f \colon X \to Y$ of integral schemes, we set $\rank(f) \coloneqq [K(X) : K(Y)]$.
	\item Given $f \colon X \to Y$ a finite morphism between Noetherian schemes, we denote by $f^{\flat}$ the right adjoint of $f_* \colon \Coh_X \to \Coh_Y$. We also write $f^! \coloneqq Rf^{\flat}$ for its right derived functor (this is consistent with the usual notation $f^!$ for $f$ a separated finite type morphism between Noetherian schemes, see \stacksproj{0A9Y}). 
	\item Given a complex $C^{\bullet}$ with values in some abelian category and $i \in \bZ$, we denote by $\cH^i(C^{\bullet})$ its $i$--th cohomology object.
	\item If $\pi \colon X \to \Spec k$ is a separated scheme of finite type over a field, we set $\omega_X^{\bullet} \coloneqq \pi^!\cO_{\Spec k}$ (see \stacksprojs{0A9Y}{0ATZ}). We also write $\omega_X \coloneqq \cH^{-\dim X}(\omega_X^{\bullet})$.
%	\item Let $X$ be a Noetherian, $(S_2)$ scheme. For any coherent reflexive sheaf of rank one $\cL$ on $X$ and $n \in \bZ$, we set $\cL^{[n]} \coloneqq (\cL^{\otimes n})^{\vee\vee}$, where by definition $\cL^{\otimes -1} \coloneqq \cL^{\vee}$. 
%	\item For any line bundle $L$ on a scheme $X$ over a field $k$, we let $\Phi_{|L|}$ denote the rational mapping induced by the linear system $|L|$.
	\item Whenever we refer to Serre duality in the context of Frobenius or Cartier modules, we are using \cite[Corollary 5.1.7]{Baudin_Duality_between_perverse_sheaves_and_Cartier_crystals}.
\end{itemize}

%STUFF:
%\begin{itemize}
%	\item maybe write $\kappa(K_X)$ and not $\kappa(X, K_X)$??? Like in Cao-Paun? TO ASK ZSOLT. OK YEP I CAN DO IT.
%\end{itemize}

\section{Preliminaries} 

%By the easy addition formula, \[ \kappa(X, mL + f^*H) \leq \kappa(F, L|_F) + \dim(Y) \] and \[ \kappa(X, mL + f^*H) \leq \kappa(X_{\eta}, L|_{X_{\eta}}) + \dim(Y). \] Now, assume that $\kappa(X, mL + f^*H) \geq 0$. By your [BCZ13, lemma 2.20], we have \[ \kappa(X, mL + f^*H) \geq \kappa(X_{\eta}, L|_{X_{\eta}}) + \dim(Y), \] so we have equality. Finally, by Proposition 1 in the other paper that you linked me (Some remarks on Kodaira dimensions of fiber spaces by Fujita), we have \[ \kappa(X, mL + f^*H) \geq \kappa(F, L|_F) + \dim(Y), \] and hence we again have equality. Thus, \[ \kappa(F, L|_F) = \kappa(X_{\eta}, L|_{X_{\eta}}). \]

\subsection{Cartier crystals, Frobenius crystals and étale $\bF_p$-sheaves}
Here, we mostly recover facts from \cite{Baudin_Duality_between_perverse_sheaves_and_Cartier_crystals}.

\begin{defn}
	Let $X$ be a Noetherian, $F$--finite (i.e. the Frobenius $F \colon X \to X$ is finite) scheme over $\bF_p$, and let $s > 0$. 
	\begin{itemize}
		\item An $s$--\emph{Cartier module} (resp. $s$--Frobenius module) is a pair ($\cM$, $\theta$) where 
		$\cM$ is an coherent $\cO_X$--module and $\theta \colon F^s_*\cM \to \cM$ (resp. $\theta \colon \cM \to F^s_*\cM$) is a morphism. We call $\theta$ the \emph{structural morphism} of $\cM$. 
		
		\item A morphism of $s$--Cartier modules (resp. $s$--Frobenius modules) is a morphism of underlying $\cO_X$-modules commuting with the $s$--Cartier module (resp. $s$--Frobenius module) structures. The category Cartier modules (resp. $s$--Frobenius modules) is denoted $\Coh_X^{C^s}$ (resp. $\Coh_X^{F^s}$).
		
		\item An $s$--Cartier module (resp. $s$--Frobenius module) is said to be \emph{nilpotent} if its structural morphism is nilpotent. These objects form a Serre subcategory of $\Coh_X^{C^s}$ (resp. $\Coh_X^{F^s}$), and the corresponding quotient category (see \stacksproj{02MS}) is denoted by $\Crys_X^{C^s}$ (resp. $\Crys_X^{F^s}$). The objects of this quotient category are denoted $s$--Cartier crystals (resp. $s$--Frobenius crystals).
		
		\item To any $s$--Cartier module (resp. $s$--Frobenius module) $(\cM, \theta)$, there is an adjoint morphism $\theta^{\flat}\colon \cM \to F^{s, \flat}\cM$ (resp. $\theta^*\colon F^{s, *}\cM \to \cM$). We call this morphism the \emph{adjoint structural morphism}. If it is an isomorphism, then $(\cM, \theta)$ is said to be \emph{unit}.
%		
%		\item A \emph{unit dualizing complex} is the datum of a dualizing complex $\omega_X^{\bullet}$, together with an isomorphism $\omega_X^{\bullet} \to F^{s, !}\omega_X^{\bullet}$. This agrees with \cite[Definition 4.2.1]{Baudin_Duality_between_perverse_sheaves_and_Cartier_crystals} by \cite[Propositions 5.1.1 and 5.1.3]{Baudin_Duality_between_perverse_sheaves_and_Cartier_crystals}.
	\end{itemize}	
\end{defn}

Fix $s > 0$. Throughout this section, we will work with $s$--Frobenius modules and $s$--Cartier modules.

\begin{rem}\label{pullback_of_Cartier_modules}
	\begin{enumerate}
		\item Since it usually does not cause any confusion, we shall simply write ``Cartier module'' and ``Frobenius module'', and omit the $s$ from the wording.
		
		\item\label{itm:Cartier_strucute_on_omega} Let $X$ be a variety over an $F$--finite field $k$ (say $f \colon X \to \Spec k$ is the structural morphism), and choose an isomorphism $\cO_{\Spec k} \to F^{s, !}\cO_{\Spec k}$. Applying $f^!$ to it gives an isomorphism $\omega_X^{\bullet} \to F^{s, !}\omega_X^{\bullet}$. Since $\cH^{-\dim X}(F^{s, !}\omega_X^{\bullet}) = F^{s, \flat}\omega_X$ (we have that $\cH^i(\omega_X^{\bullet}) = 0$ for all $i < -\dim X$), the sheaf $\omega_X$ carries the structure of a unit Cartier module.	
%		\item\label{itm:pullback_of_Cartier_modules_by_etale_maps} Let $f \colon X \to Y$ be a morphism of $F$-finite Noetherian schemes. Note that we can apply the operations $f_*$ and $f^{\flat}$ on Cartier modules. However, in general there is no reason why we could take pullbacks. 
%		
%		When the morphism is étale, this turns out to be possible. Indeed, if $\theta \colon X \to Y$ is an essentially étale morphism of schemes, then \[ \begin{tikzcd}
%			X \arrow[r, "\theta"] \arrow[d, "F"'] & Y \arrow[d, "F"] \\
%			X \arrow[r, "\theta"']                    & Y                    
%		\end{tikzcd} \] is a pullback square by \cite[XIV=XV §1 $n^{\circ}2$, Prop. 2(c)]{SGA5}, so by flat base change the natural transformation $\theta^*F_* \to F_*\theta^*$ is an isomorphism for quasi-coherent modules.
%		
%		An other equivalent way to see this is that since $\theta$ is étale, we know by \stacksproj{0WFI} that $\theta^* \cong \theta^! \cong \theta^{\flat}$. Since we can take the functor $\theta^\flat$ on Cartier modules, we win (given that both constructions define a right adjoint to $\theta_*$, they agree).
%		
%		Note from the second construction that the pullback of a unit Cartier module by an étale morphism is again unit.
		
		\item\label{itm:tensor_product_of_Cartier_module_and_unit_F_module} If $\cM$ is a Cartier module and $\cN$ is a unit Frobenius module, one can naturally endow $\cM \otimes \cN$ the structure of a Cartier module via the composition \[ F^s_*(\cM \otimes \cN) \to F^s_*(\cM \otimes F^{s, *}\cN) \to F^s_*\cM \otimes \cN \to \cM \otimes \cN. \]
		
		An important case of unit Frobenius modules are $(p^s - 1)$-torsion line bundles. Indeed, if $\alpha$ is a line bundle such that $\alpha^{p^s - 1}$ is trivial, then $F^{s, *}\alpha \cong \alpha^{p^s} \cong \alpha$.
		
%		\item\label{itm:projection_formula} When $\theta \colon X \to Y$ is an essentially étale morphism, we then obtain the usual projection formula \[ \theta_*\theta^*\cM \cong \cM \otimes \theta_*\cO_X, \] at the level of Cartier modules (the module $\theta_*\cO_X$ is a unit Frobenius module by flat base change). This follows from the analogous statement for Frobenius modules (which is immediate) and duality (see \cite[Theorem 4.2.7, Corollary 5.1.7 and Lemma 5.2.9]{Baudin_Duality_between_perverse_sheaves_and_Cartier_crystals}). IS IT OKAY? 
		
%		\item For any $s > 0$, one could work with modules $\cM$ together with a morphism $F^s_*\cM \to \cM$ (instead of only considering the case $s = 1$). The theory is identical, but since we will only need the case $s = 1$ in our applications, we do not go in this direction.
	\end{enumerate}
\end{rem}

\begin{notation}\label{notation_crystals}
	\begin{enumerate}
		\item A morphism in the category of Cartier crystals (resp. Forbenius crystals) will be written with dashed arrows, i.e. $\cM \dashrightarrow \cN$. If $\cM$ and $\cN$ are two isomorphic objects in this category, we shall write $\cM \sim_C \cN$ (resp. $\cM \sim_F \cN$).
		\item\label{itm:abuse_iterates_structure_map} Let $(\cM, \theta)$ be a Cartier module. We will abuse notations as follows: \[ \theta^e \coloneqq \theta \circ F^s_*\theta \circ \dots \circ F^{(e-1)s}_*\theta \colon F^{es}_*\cM \to \cM \] (and the analogous abuse of notations for Frobenius modules).
		\item\label{itm:notation H^0_ss} Given a proper scheme $X$ over an $F$-finite field $k$, a line bundle $L$ and a Cartier module $(\cM, \theta)$ on $X$, we set \[ H^i_{ss, \theta}(X, \cM \otimes L) \coloneqq \bigcap_{e > 0} \im\left(H^i(X, F^{es}_*\cM \otimes L) \xto{\theta^e} H^i(X, \cM \otimes L)\right) \inc H^i(X, \cM \otimes L). \] We also write $h^i_{ss, \theta}(X, \cM \otimes L) \coloneqq \dim_k H^i_{ss, \theta}(X, \cM \otimes L)$. Whenever the context is clear, we shall remove the $\theta$ from this notation (often times, we will have $L = \cO_X$).
	\end{enumerate}
\end{notation}

\begin{rem}\label{easy_semistable_when_k_perfect}
	Note that if $k$ is perfect, then $H^i_{ss}(X, \cM \otimes L) \cong \lim H^i(X, F^{es}_*\cM \otimes L)$.
\end{rem}

\begin{example}
	If $\cM$ is any Cartier module, then its structural morphism is nilpotent if and only if $\cM \sim_C 0$.
\end{example}

Either using the definitions or their adjoint versions, we see that on Frobenius modules/crystals, we can take pushforward and pullbacks. For Cartier modules/crystals, we can take pushforwards and its right adjoint for finite morphisms (see also \cite[Section 5.1]{Baudin_Duality_between_perverse_sheaves_and_Cartier_crystals} for general morphisms).

\begin{defn}
	Let $\cM$ be a Cartier module. We define \[ \Supp_{\crys}(\cM) = \set{x \in X}{\cM_x \not\sim_C 0} \inc X.\] 
\end{defn}

\begin{rem}\label{rem_support_crystal}
	 Note that the above is a closed subset (\cite[Lemma 3.3.12]{Baudin_Duality_between_perverse_sheaves_and_Cartier_crystals}), and it only depends on the underlying crystal.
\end{rem}

\begin{lem}\label{universal_homeo_induce_equivalence_of_Cartier_crystals}
	Let $f \colon X \to Y$ be a finite morphism between Noetherian $F$--finite schemes, and suppose that $f$ is a universal homeomorphism. Then the functor $f_* \colon \Crys_X^{C^s} \to \Crys_Y^{C^s}$ is an equivalence of categories, with inverse $f^{\flat}$.
	
	If $Y$ is furthermore separated and of finite type over an $F$--finite field $k$, then the natural morphism $f_*\omega_X \to \omega_Y$ of Cartier modules is an isomorphism of crystals.
\end{lem}
\begin{rem}
	Purely inseparable morphisms of normal schemes or reduction maps are examples of universal homeomorphisms (see \stacksproj{0CNF}).
\end{rem}
\begin{proof}
	The fact that $f_*$ is an equivalence of categories follows from \cite[Corollary 5.2.8]{Baudin_Duality_between_perverse_sheaves_and_Cartier_crystals}. Furthermore, we know by \cite[Corollaries 5.1.4 and 5.2.8]{Baudin_Duality_between_perverse_sheaves_and_Cartier_crystals} that both $\omega_X^{\bullet}$ and $\omega_Y^{\bullet}$ can naturally be made into complexes of unit Cartier modules, and the morphism $f_*\omega_X^{\bullet} \to \omega_Y^{\bullet}$ is an isomorphism in $D(\Crys_Y^{C^s})$. Taking cohomology sheaves gives the result.
\end{proof}

%Let us give examples of Cartier modules structures on Weil divisors:
%
%\begin{lem}
%	Assume that $X$ is normal variety over an $F$--finite field $k$, and let $D$ be a Weil divisor. Then Cartier module structures on  $\cO_X(K_X + D)$ are in one-to-one correspondence with elements $\Delta$ in the linear system $|(p^s - 1)(D)|$. Furthermore, the Cartier modules corresponding to $\Delta \in |(p^s - 1)(D)|$ and $p^j\Delta \in |(p^s - 1)(p^jD)|$ give isomorphic Cartier crystals for any $j \geq 0$.
%\end{lem}
%\begin{proof}
%	Throughout this proof, we may work in codimension two and assume that all schemes involved are regular. A Cartier module structure on $\cO_X(K_X + D)$ then corresponds to a map $\cO_X(K_X + D) \to F^{s, \flat}\cO_X(K_X + D) \cong \cO_X(K_X + p^sD)$, i.e. it corresponds to an element in $|(p^s - 1)D|$. To prove the second fact, note that if $\cO_X(K_X + D)$ has the Cartier structure corresponding to $\Delta \in |(p^s - 1)(D)|$, then $F^{j, \flat}\cO_X(K_X + D) \cong \cO_X(K_X +p^jD)$ has the Cartier structure corresponding to $p^j\Delta \in |(p^s - 1)(p^jD)|$. We then conclude by \autoref{universal_homeo_induce_equivalence_of_Cartier_crystals}.
%\end{proof}

\begin{defn}\label{def_rank_Fp_sheaf}
	Let $X$ be an irreducible, Noetherian and $F$--finite $\bF_{p^s}$--scheme.
	\begin{itemize}
		\item If $\cF$ is a constructible étale $\bF_{p^s}$-sheaf, then by definition there exists an étale open $U \to X$ such that $\cF|_U \cong \bF_{p^s, U}^{\oplus n}$. We define the \emph{rank} of $\cF$ to be $\rank(\cF) \coloneqq n$ (this makes sense since $X$ is irreducible).
		\item We define the rank of $\cF^{\bullet} \in D^b_c(X_{\et}, \bF_{p^s})$ to be \[ \rank(\cF^{\bullet}) \coloneqq \max_{i \in \bZ} \{ \rank(\cH^i(\cF^{\bullet})) \}. \]
		\item We define the \emph{rank} of a complex of Cartier crystals (resp. Frobenius--crystals) $\cM^{\bullet}$ to be the rank of the associated complex of constructible étale $\bF_{p^s}$-sheaves under the duality functor from \cite[Theorem 5.2.7]{Baudin_Duality_between_perverse_sheaves_and_Cartier_crystals} (we can apply this theorem affine locally around the generic point by \cite[Corollary 5.1.15]{Baudin_Duality_between_perverse_sheaves_and_Cartier_crystals}). We denote this quantity by $\rank_{\crys}(\cM^{\bullet})$.
	\end{itemize}
\end{defn}

\begin{rem}\label{computation_rank_integral_schemes}
	Let $\eta$ denote the generic point of $X$, and let $\eta^{\sep}$ denote a separable closure of $\eta$.
	\begin{itemize}
		\item For any $\cF \in \Sh_c(X_{\et}, \bF_{p^s})$, \[ \rank(\cF) = \dim_{\bF_{p^s}}\cF_{\eta^{\sep}}. \]
		\item The definition of the rank of a Cartier crystal does not depend on the chosen unit dualizing complex used to define the duality functor from \emph{loc. cit.}. Indeed, the rank of a Cartier crystal $\cM$ is the dimension of the unique unit Cartier module $N$ on $k(\eta)$ such that $\cM_{\eta} \sim_C N$. \\
		
		Let us show this. First, note that by \cite[Corollary 3.4.15]{Baudin_Duality_between_perverse_sheaves_and_Cartier_crystals}, two unit Cartier modules which are isomorphic as crystals are also isomorphic as Cartier modules. Furthermore, it is a straightforward computation to see that over a field, the duality functor (see \cite[Theorem 5.2.7]{Baudin_Duality_between_perverse_sheaves_and_Cartier_crystals}) exchanges unit Cartier modules and unit Frobenius modules. Since it also preserves the rank by definition, we reduce to the case of Frobenius modules. We can therefore base change and assume that $k(\eta)$ is algebraically closed. We then conclude by \cite[Corollary p.143]{Mumford_Abelian_Varieties}.
	\end{itemize}
\end{rem}

The following result will be crucial at the end of the proof of our main theorems.

\begin{lemma}\label{rank_of_pushforward}
	Let $f \colon X \to Y$ be a generically finite dominant morphism of irreducible, Noetherian and $F$--finite $\bF_{p^s}$--schemes. Then for any $\cM^{\bullet} \in D^b(\Crys_X^{C^s})$, \[ \rank_{\crys}(Rf_*\cM^{\bullet}) = [K(X_{\red}) : K(Y_{\red})]_{\sep}\rank_{\crys}(\cM^{\bullet}). \]
\end{lemma}
\begin{proof}
	By definition, it is enough to show the analogous statement for complexes of constructible $\bF_{p^s}$--sheaves $\cF^{\bullet}$. By \stacksprojs{03SI}{0CNF}, we may assume that $X$ and $Y$ are reduced. Since this is a local question on $Y$, we may further assume that $f$ is finite. In particular, $Rf_* = f_*$, so by definition it is enough to show the result for $\bF_{p^s}$-sheaves (and not any complex). By \autoref{computation_rank_integral_schemes}, we may further assume that $X = \Spec K(X)$ and $Y = \Spec K(Y)$. Since inseparable extensions induce universal homeomorphisms (see \stacksproj{0CNF}), we may assume $K(Y) \inc K(X)$ is separable. Now, let $\cF$ be an étale $\bF_{p^s}$-sheaf on $\Spec(K(Y))$, and let $L$ be a separable field extension of $K(Y)$ such that $\cF|_{\Spec L} \cong \bF_{p^s}^{\oplus n}$. Up to making $L$ bigger, we may assume that $L$ contains a Galois closure of $K(Y) \inc K(X)$. We then obtain that
	\begin{align*} 
		H^0_{\et}(\Spec L, f_*\cF|_{\Spec L}) & = H^0_{\et}(\Spec (L \otimes_{K(Y)} K(X)), \cF|_{\Spec (L \otimes_{K(X)} K(Y))})
		\\ & \expl{\cong}{$L \otimes_{K(Y)} K(X) \cong L^{\times [K(X) : K(Y)]}$} H^0_{\et}(\Spec L, \cF|_{\Spec L})^{\oplus [K(X) : K(Y)]} \\
		& \cong \bF_{p^s}^{\oplus n\cdot [K(Y) : K(X)]}. \qedhere
	\end{align*}
\end{proof}

\subsection{Frobenius stable Kodaira dimension}

Throughout this section, $k$ denotes a fixed $F$--finite field. As in \autoref{pullback_of_Cartier_modules}.\autoref{itm:Cartier_strucute_on_omega}, we further choose an isomorphism $\cO_{\Spec k} \to F^!\cO_{\Spec k}$, so that for any $s > 0$, any separated and finite type $k$--scheme has the structure of a unit $s$--Cartier module on its dualizing module. We also fix a normal proper variety over $k$, an effective $\bZ_{(p)}$-divisor $\Delta$ on $X$, and $\delta \in \{\pm 1\}$. Recall the following definition:
\begin{defn}\label{def:kod_stable_kod_dim}
	Let $\cM$ be a reflexive sheaf on $X$.
	\begin{itemize}
		\item Define $S^0(X, \Delta; \cM) \inc H^0(X, \cM)$ to be \[ \bigcap_{e \gg 0} \im\left(\Tr^e_{\Delta} \colon H^0\left(X, \left(F^e_*\cO_X((1 - p^e)(K_X + \Delta)) \otimes \cM\right)^{\vee\vee} \right) \to H^0(X, \cM)\right), \]
		where $\Tr^e_{\Delta}$ is induced by the Grothendieck trace map \[ \begin{tikzcd}
			F^e_*\cO_X((1 - p^e)(K_X + \Delta)) \arrow[rr, "\inc"] &  & F^e_*\cO_X((1 - p^e)K_X) \arrow[rr] &  & \cO_X
		\end{tikzcd} \] by twisting with $\cM$ and reflexifying. Note that $\Tr^e_{\Delta}$ is an abuse of notation as in 	\autoref{notation_crystals}.\autoref{itm:abuse_iterates_structure_map}.
		\item If $D$ is a $\bQ$-divisor on $X$ and $S^0(X, \Delta; \cO_X(nD)) \neq 0$ for some $n > 0$, we set \[ \kappa_S(X, \Delta; D) \coloneqq \min\set{j \geq 0}{\dim S^0(X, \Delta; \cO_X(mD)) = O(m^j) \mbox{ for $m > 0$ sufficiently divisible} }. \]  Otherwise, we set $\kappa_S(X, \Delta; D) = -\infty$. When $D = \delta(K_X + \Delta)$, we will simply write $\kappa_S(\delta(K_X + \Delta)) \coloneqq \kappa_S(X, \Delta; \delta(K_X + \Delta))$.
	\end{itemize}
\end{defn}

\begin{rem}
	\begin{itemize}
		\item To give an idea, $S^0(X, \Delta; \cO_X) \neq 0$ if and only if the pair $(X, \Delta)$ is globally sharply $F$--split (see \cite[Definition 3.1]{Schwede_Smith_Globally_F_regular_and_log_Fano_varieties}).
		\item In the literature, $S^0(X, \Delta; \cO_X(D))$ is also often denoted $S^0(X, \sigma(X, \Delta) \otimes \cO_X(D))$.
	\end{itemize}
\end{rem}

\begin{lem}\label{K_S(X)_not_-infty_implies_K_S(X)_equals_K(X)}
	Let $D$ be a divisor on $X$. Then the subset \[ S(X, \Delta; \cO_X(D)) \coloneqq \bigoplus_{m \geq 0}S^0(X, \Delta; \cO_X(mD)) \inc \bigoplus_{m \geq 0} H^0(X, \cO_X(mD)) \eqqcolon R(X, \cO_X(D)) \] is an ideal of $R(X, \cO_X(D))$. In particular, if $\kappa_{S}(X, \Delta; D) \neq -\infty$, then $\kappa_{S}(X, \Delta; D) = \kappa(X, D)$.
\end{lem}
\begin{proof}
	The proof is identical to that of \cite[Lemmas 4.1.1 and 4.1.3]{Hacon_Pat_GV_Characterization_Ordinary_AV}.
\end{proof}

Let us define a Cartier module structure on $\cO_X(\delta m(K_X + \Delta))$, whose space of stable sections is a subspace of $S^0(X, \Delta; \cO_X(\delta m(K_X + \Delta)))$.

\begin{defn}\label{Cartier_structure_on_pluricanonical_bundles}
	Fix an element $t \in H^0(X, \cO_X(\delta r(K_X + \Delta)))$ for some $r > 0$ coprime to $p$ such that $r(K_X + \Delta)$ is integral. Let $s > 0$ be an integer such that $r$ divides $p^s - 1$. For any $m > 0$ such that $m(K_X + \Delta)$ is integral, define $\theta_t \colon F^s_*\cO_X(\delta m(K_X + \Delta)) \to \cO_X(\delta  m(K_X + \Delta))$ by the composition
	\[  \begin{tikzcd}
		F^s_*\cO_X(\delta m(K_X + \Delta)) \arrow[rr, "F^s_*t^{(m - \delta )\frac{(p^s - 1)}{r}}"] &  & F^s_*\cO_X((1 - p^s + \delta  p^sm)(K_X + \Delta)) \arrow[r, "\Tr^s_{\Delta}"] &  \cO_X(\delta m(K_X + \Delta)),
	\end{tikzcd} \] where we recall that $\Tr^s_{\Delta}$ is the twist by $\cO_X(\delta m(K_X + \Delta))$ of the composition 	
	\[  \begin{tikzcd}
		F^s_*\cO_X((1 - p^s)(K_X + \Delta)) \arrow[rr, "\inc"] &  & 	F^s_*\cO_X((1 - p^s)K_X) \arrow[r, "\Tr^s"] &  \cO_X.
	\end{tikzcd} \]
	A straightforward computation shows that \[ H^0_{ss, \theta_t}(X, \cO_X(\delta m(K_X + \Delta))) \inc S^0(X, \Delta; \cO_X(\delta m(K_X + \Delta))). \]
\end{defn}

We conclude this subsection with a useful result from \cite{Hacon_Pat_GV_Characterization_Ordinary_AV}. We set $D = \delta (K_X + \Delta)$, and we fix integers $r$, $s$, $m$ and a section $0 \neq t \in H^0(X, \cO_X(rD))$ as in \autoref{Cartier_structure_on_pluricanonical_bundles}.

\begin{lemma}\label{basic_props_var_kod_zero}
	Assume that $\kappa_{S}(X, \Delta; D) = 0$. Then \[H^0_{ss, \theta_t}(X, \cO_X(mD)) = S^0(X, \Delta; \cO_X(mD)) = H^0(X, \cO_X(mD)). \]
\end{lemma}
\begin{proof}
	This is already present in \cite[Lemma 4.2.5]{Hacon_Pat_GV_Characterization_Ordinary_AV} in a less general form. Although their proof essentially applies to this setup, we reprove it here for the sake of the reader.
	
	Assume that $H^0(X, \cO_X(mD)) \neq 0$, otherwise there is nothing to do. Recall that $\theta_t$ is given by \[  \begin{tikzcd}
		F^s_*\cO_X(mD) \arrow[rr, "F^s_*t^{(m - \delta )\frac{(p^s - 1)}{r}}"] &  & F^s_*\cO_X((\delta (1 - p^s) + p^sm)D) \arrow[r, "\Tr^s_{\Delta}"] &  \cO_X(mD),
	\end{tikzcd} \] so since the first arrow is an inclusion, it induces an inclusion of global sections as well. Given that $\kappa(X, D) = 0$ by \autoref{K_S(X)_not_-infty_implies_K_S(X)_equals_K(X)}, we deduce that this inclusion is in fact an equality, whence \[ H^0_{ss, \theta_t}(X, \cO_X(mD)) = S^0(X, \Delta; \cO_X(mD)). \]
	
	Now, let $e \gg 0$. We are left to show that $\Tr^{es}_{\Delta}$ is not zero on global sections. Let $0 \neq f \in H^0(X, \cO_X(mD))$, and let $l' > 0$ be an integer such that $S^0(X, \Delta; \cO_X(l'mD)) \neq 0$. By definition, $\Tr^{es}_{\Delta}$ induces a surjection \[ F^{es}_*S^0\left(X, \Delta; \cO_X\left((\delta (1 - p^{es}) + p^{es}l'm)D\right)\right) \surj S^0(X, \Delta; \cO_X(l'mD)), \] so $l'' \coloneqq \delta (1 - p^{es}) + p^{es}l'm$ satisfies $S^0(X, \Delta; \cO_X(l''D)) \neq 0$. Note that $p^{es}$ divides $l'' - \delta$ by construction, so $l \coloneqq (l'')^2$ is a positive integer such that 
	\[ \begin{cases*}
		p^{es} \mbox{ divides } l - 1 \\
		S^0(X, \Delta; \cO_X(lmD)) \neq 0.
	\end{cases*} \]
	Since $\kappa(X, D) = 0$, there exists some $\lambda \in k^\times$ such that \[ \Tr_{\Delta}^{es}\left(F^{es}_*\left(\lambda t^{(lm - \delta )\frac{p^{es} - 1}{r}}f^l\right)\right) = f^l. \] The left--hand side can be rewritten as \begin{align*}
		\Tr_{\Delta}^{es}\left(F^{es}_*\left(\lambda t^{(lm - \delta )\frac{p^{es} - 1}{r}}f^l\right)\right) & = \Tr_{\Delta}^{es}\left(F^{es}_*\left(\left(\lambda t^{(m - \delta )\frac{p^{es} - 1}{r}}f\right) \cdot \left(t^{\frac{(l - 1)m(p^{es} - 1)}{rp^{es}}}f^{\frac{l - 1}{p^{es}}}\right)^{p^{es}}\right)\right) \\
		& = \Tr_{\Delta}^{es}\left(F^{es}_*\left(\lambda t^{(m - \delta )\frac{p^{es} - 1}{r}}f\right)\right)t^{\frac{(l - 1)m(p^{es} - 1)}{rp^{es}}}f^{\frac{l - 1}{p^{es}}},
	\end{align*}
	so \[ \Tr_{\Delta}^{es}\left(F^{es}_*\left(\lambda t^{(m - \delta )\frac{p^{es} - 1}{r}}f\right)\right) \neq 0. \] This is exactly what we wanted to show.
\end{proof}

\subsection{Abelian varieties, $V$--modules and generic vanishing}

\emph{From now on until the end of the paper, we work over a fixed algebraically closed base field $k$}. We also fix an integer $s > 0$. Our objective here is to gather all the results about abelian varieties and generic vanishing in positive characteristic we need to carry the proof of our main theorems. First, let us recall the existence theorem of the Albanese morphism.

\begin{thm}\label{existence_of_albanese}
	Let $X$ be a proper $k$--scheme such that $H^0(X, \cO_X) = k$. Then there exists a unique morphism $\alb_X \colon X \to \Alb(X)$ to an abelian variety $\Alb(X)$ such that for any other morphism $g \colon X \to B$ to an abelian variety, there exists a unique morphism $h \colon \Alb(X) \to Q$ such that $h \circ \alb_X = g$.
	
	Furthermore, the pullback of line bundles map $\Pic^0_{\Alb(X)/k} \to \Pic^0_{X/k}$ is a closed immersion, and the pullback of sections map $H^1(\Alb(X), \cO_{\Alb(X)}) \to H^1(X, \cO_X)$ is injective. Finally, $b_1(X) = 2\dim(\Alb(X))$.
\end{thm}
\begin{proof}
	Everything except the two statements follows from \cite[Definition 8.1 and Theorem 10.2]{Schroer_Laurent_Para_Abelian_varieties_and_Albanese_maps} (with their notation, a para--abelian variety over an algebraically closed field is already an abelian variety by \cite[First paragraph in Section 4 and Proposition 4.3]{Schroer_Laurent_Para_Abelian_varieties_and_Albanese_maps}). Since the tangent space at $\cO_X$ of $\Pic^0_{X/k}$ is exactly $H^1(X, \cO_X)$, and similarly for $\Alb(X)$ instead of $X$ (see for example \cite[Theorem 9.5.11]{Kleiman_The_Picard_Scheme} and \cite[Theorem 2.1]{Schroer_Laurent_Para_Abelian_varieties_and_Albanese_maps}), the statement about $H^1$--groups follows. Finally, see for example \cite[Lemma 1.6]{Roessler_Schroer_Varieties_with_free_tangent_sheaves}) for the equality $b_1(X) = 2\dim(\Alb(X))$.
\end{proof}

Throughout, $A$ denotes an abelian variety over $k$ of dimension $g$, with dual abelian variety $\bighat{A}$ and normalized Poincaré bundle $\cP \in \Pic(A \times \bighat{A})$. Although $\omega_A \cong \cO_A$, we shall write $\omega_A$ when we want to see it as a Cartier module and $\cO_A$ when we see it as an Frobenius module.

\begin{defn}\label{def:property_AV}
	\begin{itemize}
		\item The \emph{$p$--rank} of $A$ is the biggest integer $i \in \{0, \dots, g\}$ such that $H^i(X, \cO_A) \not\sim_F 0$. 
		\item If the $p$--rank of $A$ is $g$, then $A$ is said to be \emph{ordinary}.
		\item An abelian variety is said to be \emph{supersingular} if it is isogeneous to a product of non--ordinary elliptic curves. 
		\item An abelian variety is said to have \emph{no supersingular factor} if there exists no non--trivial morphism from a supersingular abelian variety (or equivalently from a supersingular elliptic curve).
	\end{itemize}
\end{defn}

\begin{rem}
	Let $r(A)$ denote the $p$-rank of $A$. Since $H^*(A, \cO_A) = \bigwedge H^1(A, \cO_A)$, we deduce that the $p$--rank of $A$ is also the dimension of the semistable part of $H^1(A, \cO_A)$. In other words $A$ has exactly $p^{r(A)}$-torsion points by \cite[top of p.148]{Mumford_Abelian_Varieties}. Our definition therefore agrees with the usual one.
\end{rem}

\begin{prop}\label{basic facts about AV}
	Let $r(A)$ denote the $p$--rank of $A$. Then the following holds: 
	\begin{enumerate}
		\item\label{itm:p_rank stable under isogenies} two isogeneous abelian varieties have the same $p$--rank;
		\item\label{itm:A_surjects_to_B_and_A_ord_implies_B_ord} if $f \colon A \to B$ is a surjective morphism of abelian varieties and $A$ is ordinary, then so is $B$;
		\item\label{itm:ordinary_iff_V_etale} the (absolute) Verschiebung $V$ (see \autoref{def_V_module}) is étale if and only if $A$ is ordinary.
	\end{enumerate}
\end{prop}
\begin{proof}
	See \cite[p. 147]{Mumford_Abelian_Varieties} and \cite[Lemma 2.3.4]{Hacon_Pat_GV_Characterization_Ordinary_AV}.
\end{proof}

\begin{defn}[{\cite[Section 5]{Hacon_Pat_GV_Geom_Theta_Divs}}]\label{def_V_module}
	The \emph{relative Verschiebung} $V_{A/k} \colon \bighat{A}^{(p)} = \bighat{A^{(p)}} \to \bighat{A}$ is by definition the dual isogeny of the relative Frobenius $F_{A/k} \colon A \to A^{(p)}$ (see \stacksproj{0CC9}). Since $k$ is perfect, we can canonically identify $\bighat{A}^{(p)}$ and $\bighat{A}$ as schemes. The induced morphism $V \colon \bighat{A} \to \bighat{A}$ is called the \emph{absolute Verschiebung}. By definition, the diagram
	\[ \begin{tikzcd}
		\bighat{A} \arrow[rr, "V"] \arrow[d] &  & \bighat{A} \arrow[d] \\
		\Spec k \arrow[rr, "F^{-1}"]         &  & \Spec k             
	\end{tikzcd} \] commutes.
	\begin{itemize}
		\item A \emph{$V^s$--module} on $\bighat{A}$ is a pair $(\cN, \theta)$ where $\cN$ is a coherent sheaf on $\bighat{A}$ with a morphism $\theta \colon \cN \to V^{s, *}\cN$. A morphism of $V$--modules is a morphism of underlying sheaves commuting with the $V^s$--module structures. The category of $V^s$-modules on $\bighat{A}$ is denoted $\Coh_{\bighat{A}}^{V^s}$.
		\item A $V^s$--module $(\cM, \theta)$ is nilpotent if $\theta^N = 0$ for some $N \geq 0$ (we abuse notation as in \autoref{notation_crystals}.\autoref{itm:abuse_iterates_structure_map}). Nilpotent modules form a Serre subcategory, whose quotient is the category of \emph{$V^s$--crystals}, denoted $\Crys_{\bighat{A}}^{V^s}$. Morphisms of $V^s$-crystals are denoted by $\dra$, and if $\cM$ and $\cN$ are two isomorphic $V^s$--crystals, we write $\cM \sim_V \cN$. 
		\item A $V^s$--module $(\cM, \theta)$ is said to be \emph{injective} if $\theta$ is injective.
	\end{itemize}
\end{defn}
\begin{rem}\label{canonical_injective_V_module}
	\begin{itemize}
		\item As for Frobenius and Cartier modules, we will omit the $s$ in the name ``$V^s$--module''.
		\item For any $V$--module $\cN$, there exists a canonical injective $V$--module $\cN_{\injj}$ with a surjective map $\cN \surj \cN_{\injj}$ inducing an isomorphism of crystals. Explicitly, $\cN_{\injj}$ is given by the image of $\cN \to \colim V^{es, *}\cN$. Most of the time, we will only consider injective $V$--modules.
	\end{itemize}
\end{rem}

%
%\begin{lem}\label{nil-iso iff iso after unitalization}
%	Let $f \colon \cM \ra \cN$ be a morphism of $V$--modules. Then $f$ induces an isomorphism of $V$-crystals if and only if the induced map \[ \colim V^{es, *}\cM \ra  \colim V^{es, *}\cN \] is an isomorphism.
%\end{lem}
%\begin{proof}
%	A morphism of $V$--modules induces an isomorphism of crystals if and only if both its kernel and cokernel are nilpotent. Thus, the result follows from the facts that $V^*$ is exact, taking filtered colimits is exact, and a $V$--module $\cP$ is nilpotent if and only if $\colim V^{es, *} \cP = 0$.
%\end{proof}

Recall the following definition from \cite{Mukai_Fourier_Mukai_transform} (see also \cite{Schnell_Fourier_Mukai_transform_made_easy}).
\begin{defn}
	\begin{itemize}
		\item The \emph{Fourier--Mukai transforms} are the functors $R\bighat{S} \colon D_{\coh}(A) \to D_{\coh}(\bighat{A})$ and $RS \colon D_{\coh}(\bighat{A}) \to D_{\coh}(A)$ defined by 
		\[ R\hat{S}(\cM) \coloneqq Rp_{\hat{A}, *}(Lp_A^*\cM \otimes \cP)  \esp \mbox{ and } \esp RS(\cN) \coloneqq Rp_{A, *}(Lp_{\bighat{A}}^*\cN \otimes \cP), \] where $p_A \colon A \times \bighat{A} \to A$ and $p_{\bighat{A}} \colon A \times \bighat{A} \to \bighat{A}$ denote the projections.
		\item The \emph{symmetric Fourier--Mukai transforms} are the functors $\FM_A \coloneqq R\bighat{S} \circ D_A$ and $\FM_{\bighat{A}} \coloneqq RS \circ D_{\bighat{A}}$, where for a separated $k$--scheme $Y$, we set $D_Y \coloneqq \cR\HHom(-, \omega_Y^{\bullet})$.
	\end{itemize}
\end{defn}

The Fourier-Mukai transform has the following important properties, which are central in generic vanishing theory.

\begin{thm}\label{properties_Fourier_Mukai_transform}
	Let $\cM$ be a coherent sheaf on $A$.
	\begin{enumerate}
		\item\label{itm:support} For all $i \notin \{-g, \dots, 0\}$, then $\cH^i(\FM_A(\cM)) = 0$ for all $i \notin \{-g, \dots 0\}$. The same holds for $\FM_{\bighat{A}}$.
		\item\label{itm:equiv_cat} The functors $\FM_A$ and $\FM_{\bighat{A}}$ are inverses of each other, and hence equivalences of categories.
		\item\label{itm:equiv_cat_unipotent} We have $\FM_A(\omega_A) = k(0)$ (i.e. the skyscraper sheaf at $0 \in \bighat{A}$). In particular, the functor $\FM_A$ induces an equivalence of categories between unipotent vector bundles (i.e. vector bundles which are consecutive extensions of $\omega_A$) and coherent modules on $\bighat{A}$ supported at the point $0$. More generally, it also induces an equivalence of categories between homogeneous (i.e. translation--invariant) vector bundles on $\bighat{A}$ and finite length modules on $\bighat{A}$.
		\item\label{itm:behaviour_under_translation} For any $\alpha \in \bighat{A}(k)$, \[ T_{-\alpha}^*  \circ \FM_A \cong \FM_A \circ \: (\alpha \otimes -),\] where $T_\alpha$ denotes the translation map by $\alpha$.
		\item\label{itm:behaviour_pushforwards_and_pullbacks} Let $f \colon A \to B$ be a morphism of abelian varieties, with dual morphism $\bighat{f} \colon \bighat{B} \to \bighat{A}$. Then we have \[ \FM_B \: \circ \: Rf_* \cong L\bighat{f}^* \circ \FM_A. \]
		In particular, if $\cM$ is a Cartier module, then each $\cH^i(\FM_A(\cM))$ has an induced $V$--module structure and vice--versa.
		\item\label{itm:support_H0} We have that \[ \Supp \left(\cH^0\FM_A(\cM)\right) = \set{ \alpha \in \bighat{A}}{ H^0(A, \cM \otimes \alpha^{\vee}) \neq 0}.\]
	\end{enumerate}
\end{thm}
\begin{proof}
	Point \autoref{itm:support} follows from the first lines in the proof of \cite[Theorem 3.1.1]{Hacon_Pat_GV_Characterization_Ordinary_AV}. Points \autoref{itm:equiv_cat}, \autoref{itm:behaviour_under_translation} and \autoref{itm:behaviour_pushforwards_and_pullbacks} follow respectively from \cite[Theorem 3.2, Propositions 4.1 and 5.1]{Schnell_Fourier_Mukai_transform_made_easy}. Point \autoref{itm:equiv_cat_unipotent} can be found in \cite[Examples 2.9 and 3.2]{Mukai_Fourier_Mukai_transform}. Finally, point \autoref{itm:support_H0} follows from \cite[Theorem 2.2.8.(g)]{Baudin_Positive_characteristic_generic_vanishing_theory}.
\end{proof}

The following is a collection of central results in the positive characteristic side of generic vanishing theory.

\begin{thm}\label{equivalence_V_crystals_and_Cartier_crystals}
	The functor $(\Coh_A^{C^s})^{op} \to \Coh_{\bighat{A}}^{V^s}$ given by $\cM \mapsto \cH^0(\FM_A(\cM))$ induces an equivalence of categories \[ (\Crys_A^{C^s})^{op} \cong \Crys_{\bighat{A}}^{V^s}. \]
	Furthermore, if $\cM \in \Coh_A^{C^s}$, then for all $i \neq 0$, \[ \cH^i(\FM_A(\cM)) \sim_V 0, \] and the analogous statement for $V$--modules also holds. In particular, if we let $\cN \coloneqq \cH^0(\FM_A(\cM))$, then for all $\alpha \in \Pic^0(A)$ and $i \geq 0$, we have \[ \varprojlim H^i(A, F^{es}_*\cM \otimes \alpha)  \cong \left(\colim \Tor_i(V^{es, *}\cN, k(- \alpha))\right)^{\vee}.  \] This also holds if we replace $\cN$ by $\cN_{\injj}$.
\end{thm}
\begin{proof}
	See \cite[Proposition 3.2.5, Theorems 3.1.2 and 3.2.1 and Lemma 3.3.9]{Baudin_Positive_characteristic_generic_vanishing_theory} (see also \cite[Theorem 1.3.1]{Hacon_Pat_GV_Characterization_Ordinary_AV} and \cite[Theorem 5.2]{Hacon_Pat_GV_Geom_Theta_Divs}).
\end{proof}

Let us now setup some notations in order to phrase (a version of) the next important result.

\begin{notation}
	Let $\cM$ be a Cartier module on $A$, let $\cN \coloneqq \cH^0\FM_A(\cM)$, and let $i \geq 0$. We set \[ W^i_F(\cM) \coloneqq \set{\alpha \in \Pic^0(A)}{\lim H^i(A, F^{es}_*\cM \otimes \alpha) \neq 0} \subseteq \Pic^0(A), \] and \[ V^i_{\injj}(\cM) \coloneqq \set{\alpha \in \bighat{A}}{\Tor_i(\cN_{\injj}, k(-\alpha)) \neq 0} \inc \bighat{A}. \]
	Given any coherent sheaf $\cM'$ on $A$, we also define $V^0(\cM') \coloneqq -\Supp \cH^0\FM_A(\cM') \inc \bighat{A}$ (see \autoref{properties_Fourier_Mukai_transform}.\autoref{itm:support_H0}).
\end{notation}

The idea is that the closed subsets $V^i_{\injj}(\cM)$ are supposed to approximate well the $W^i_F(\cM)$. Let us make this precise (see \cite[Theorem 3.3.5 and Proposition 3.3.17]{Baudin_Positive_characteristic_generic_vanishing_theory} for a more complete version of this theorem and Section 3.4 from \emph{loc. cit.} for examples):

\begin{thm}\label{generic vanishing}
	Let $\cM$ be a Cartier module on $A$, and let $i \geq 0$. Then the following holds:
	\begin{enumerate}[topsep=1ex, itemsep=1ex]
		\item\label{itm:gv_inclusions} $V^{i + 1}_{\injj}(\cM) \inc V^i_{\injj}(\cM)$;
		\item\label{itm:gv_codim} $\codim V^i_{\injj}(\cM) \geq i$;
		\item\label{itm:gv_p_closed} if $S^i(\cM)$ denotes the union of irreducible components of $V^i_{\injj}(\cM)$ of codimension $i$, then $p^s(S^i(\cM)) \inc S^i(\cM)$. In particular, prime--to--$p$ torsion points are dense in $S^i(\cM)$ and if $A$ has no supersingular functor, then every irreducible component of $S^i(\cM)$ is a translate of an abelian subvariety by a prime--to--$p$ torsion point;
		\item\label{itm:gv_link_with_section}
		we have the inclusion $V^0_{\injj}(\cM) \inc V^0(\cM)$;
		\item\label{van_and_non_van} a very general closed point of $S^i(\cM)$ is in $W^i_F(\cM)$, and for any $\alpha \in W^i_F(\cM)$, there exists $e \geq 0$ such that $p^{es}(\alpha) \in V^i_{\injj}(\cM)$.
	\end{enumerate}
\end{thm}

\begin{proof}
	See \cite[Theorems 3.3.5 and 3.3.7, Proposition 3.3.17]{Baudin_Positive_characteristic_generic_vanishing_theory} and \cite[Proposition 6.1]{Pink_Roessler_Manin_Mumford_conjecture} (with their terminology, $\bighat{A}$ is positive since $p^{es} \colon \bighat{A} \to \bighat{A}$ is never an isomorphism on a non--trivial abelian variety). See also \cite[Theorem 1.1]{Hacon_Pat_GV_Geom_Theta_Divs}.
\end{proof}

\begin{rem}\label{exists_V^i_with_codim_i}
	Note that for any Cartier module $\cM$, there exists $i \geq 0$ such that $V^i_{\injj}(\cM)$ has an irreducible component of codimension $i$. Indeed, if $y \in \bighat{A}$ is an associated point of $\cH^0(\FM_A(\cM))_{\injj}$ of codimension $i$, then $-y \in V^i_{\injj}(\cM)$ by definition.
\end{rem}

From now on, we will rather use the shorter notation $H^i_{ss}(A, \cM \otimes \alpha)$ instead of $\lim H^i(A, F^{es}_*\cM \otimes \alpha)$ (see \autoref{easy_semistable_when_k_perfect}). We will need the following in the main proof.

\begin{lem}\label{very_gen_pts_in_V0_implies_in_V0}
	Assume that $k$ is uncountable, let $r \geq 0$, let $\cM$ be a Cartier module on $A$ and let $Z \inc \bighat{A}$ be an irreducible closed subset such that for very general $\alpha \in Z$, \[ H^r_{ss}(A, \cM \otimes \alpha) \neq 0. \] Then $p^{es}(Z) \inc V^r_{\injj}(\cM)$ for some $e \geq 0$. In particular, $p^{es}(Z) \inc V^0_{\injj}(\cM)$.
\end{lem}
\begin{proof}
	Given $e \geq 0$, let $Z^e \coloneqq (p^{es})^{-1}(V^r_{\injj}(\cM)) \cap Z$. We then obtain that $\bigcup_{e \geq 0} Z^e$ contains a very general point of $Z$ by assumption and \autoref{generic vanishing}.\autoref{van_and_non_van}. Note that $\bigcup_{e \geq 0} Z^e$ is a countable union of closed subsets, so we must have that $Z^e = Z$ for some $e \geq 0$ (recall that $Z$ is irreducible). In particular, $p^{es}(Z) \inc V^r_{\injj}(\cM)$. The last statement follows, since $V^r_{\injj}(\cM) \inc V^0_{\injj}(\cM)$ (see \autoref{generic vanishing}.\autoref{itm:gv_inclusions}).
\end{proof}

It follows from the definitions that if $f \colon A \to B$ is a morphism of abelian varieties and $\cN$ is a $V$--module on $B$, then $f^*\cN$ is a $V$--module on $A$. Let us discuss how to take pushforwards and upper shriek functors too, and what this corresponds to for Cartier modules.

\begin{lem}\label{restriction_of_a_V_module_to_subschemes}
	There exists an isomorphism of functors $V^\flat \cong V^*$. In particular, given a $V$--module $\cM$, the following holds:
	\begin{itemize}
		\item there is an induced adjoint structural morphism $V^s_*\cM \to \cM$, so one can take pushforwards of $V$--modules;
		\item given a finite morphism of abelian varieties $i \colon A \to B$ and a $V$--module $\cM$ on $A$, the object $i^{\flat}\cM$ is a $V$--module on $B$.
		\item we have the usual adjunction $i_* \dashv i^{\flat}$ for $V$--modules.
	\end{itemize}
\end{lem}
\begin{proof}
	By \stacksproj{0E4K}, it is enough to show that $V^\flat\cO_A \cong \cO_A$ to prove that $V^{\flat} \cong V^*$. This former fact follows
	from the commutativity of the diagram 
	\[ \begin{tikzcd}
		A \arrow[r, "V"] \arrow[d, "\pi"'] & A \arrow[d, "\pi"] \\
		\Spec k \arrow[r, "F^{-1}"]        & \Spec k,        
	\end{tikzcd} \] together with the facts that $\pi^!\cO_{\Spec k} = \cO_A[g]$ and that $(F_k^{-1})^!\cO_{\Spec k} = \cO_{\Spec k}$. The statements after ``In particular'' are immediate consequences.
\end{proof}
\begin{rem}\label{rem:upper_shriek_and_pullback_are_equal_for_isogeny}
	A similar argument shows that for any isogeny $\phi \colon A \to B$ of abelian varieties, there is an isomorphism $\phi^* \cong \phi^{\flat}$. The same also holds for the Frobenius, so we can also take pullbacks of Cartier modules under morphisms between abelian varieties.
\end{rem}

	Note that there is no preferred isomorphism $V^{\flat} \cong V^*$ a priori, so our construction above seems to depend on choices. We will discuss this subtelty after the proof of \autoref{basic facts about AV_not_so_basic}, but for now let us draw some properties and consequences from this construction.

\begin{lem}\label{support_of_pullback}
	Let $\pi \colon A \to B$ be a surjective morphism of abelian varieties, and let $\cM$ be a Cartier module on $B$. Then $\Supp_{\crys}(\pi^*\cM) = \pi^{-1}(\Supp_{\crys}(\cM))$.
\end{lem}
\begin{proof}
	This follows by faithful flatness of $\pi$.
\end{proof}

As an application of taking the functor $i^{\flat}$ on $V$--modules, we have the following:

\begin{lem}\label{filtration_for_V_modules_supp_at_0}
	Let $\cN$ be a $V$--module on $\bighat{A}$ supported at $0$. Then there is a filtration $0 = \cN_0 \inc \cN_1 \inc \dots \inc \cN_n = \cN$ of $V$--modules, such that for all $0 \leq i \leq n - 1$, $\cN_{i + 1}/\cN_i \cong k(0)$, endowed with either the zero $V$--module structure or with its canonical $V$--module structure (coming from the Cartier structure of $\omega_A$ via $\FM_A$).
\end{lem}
\begin{proof}
	Throughout, we will use \autoref{restriction_of_a_V_module_to_subschemes} without further mention. Recakk that $i_0 \colon \Spec k(0) \to \bighat{A}$ denote the inclusion of 
	$0 \in \bighat{A}$. In order to avoid confusions, for this argument only, we shall write $i_{0, *}k(0)$ to denote the module on $\bighat{A}$, while $k(0)$ stands for the module on the field $k(0) = k$.
	
	We show this by induction on the length of $\cN$. We have an exact sequence \[ 0 \to i_{0, *}i_0^{\flat}\cN \to \cN \to \cC \to 0, \] where $\length(\cC) < \length(\cN)$. The $V$--module $i_{0, *}i_0^{\flat}\cN$ is the pushforward of a $V$--module on $k(0)$, i.e. a Frobenius module on $k(0)$. By the structure theorem of Frobenius modules over an algebraically closed field (see \cite[Corollary p.143]{Mumford_Abelian_Varieties}), we know that $i_0^{\flat}\cN \cong (k(0), \phi)^{\oplus n_1} \oplus (k(0)^{\oplus n_2}, \psi)$, where $\phi$ is the canonical Frobenius module structure on $k(0)$, and $\psi$ is nilpotent. Note that $(k(0)^{\oplus n_2}, \psi)$ has a filtration by sub--Frobenius modules, where each quotient is the trivial Frobenius module $(k(0), 0)$. 
	
	Thus, the result holds for $i_{0, *}i_0^{\flat}\cN$. Since it also holds for $\cC$ by induction, the proof is complete.
\end{proof}

The following lemma will be used in the main proof to prove that certain abelian varieties are ordinary. Recall that the $p$--rank of $A$ is denoted $r(A)$.

\begin{lem}\label{basic facts about AV_not_so_basic}
	For any integer $0 \leq j \leq g$, the following are equivalent:
	\begin{enumerate}
		\item\label{basic_facts_AV_itm:there_exists} there exists a $V$--module $\cN$ such that $\Supp(\cN) = \{0\}$ and \[ \colim \Tor_j(V^{es, *} \cN, k(0)) \neq 0; \] 
		\item\label{basic_facts_AV_itm:for_all} for any non-nilpotent $V$--module $\cN$ such that $\Supp(\cN) = \{0\}$, we have \[ \colim \Tor_j(V^{es, *} \cN, k(0)) \neq 0; \] 
		\item\label{basic_facts_AV_itm:p_rank} $r(A) \geq g - j$.
	\end{enumerate}
\end{lem}
\begin{proof}
	The implication \autoref{basic_facts_AV_itm:for_all} $\implies$ \autoref{basic_facts_AV_itm:there_exists} is immediate. Let us show that \autoref{basic_facts_AV_itm:p_rank} $\implies$ \autoref{basic_facts_AV_itm:for_all}. Thanks to \autoref{filtration_for_V_modules_supp_at_0} and long exact sequences in homology, it is enough to show the result when $\cN = k(0)$, together with its canonical $V$--module structure. This is contained in \cite[Example 3.4.1]{Baudin_Positive_characteristic_generic_vanishing_theory}.
	
	We are left to show that \autoref{basic_facts_AV_itm:there_exists} $\implies$ \autoref{basic_facts_AV_itm:p_rank}. Consider a filtration as in \autoref{filtration_for_V_modules_supp_at_0}. Again by the long exact sequence in homology, we obtain that \[ \colim \Tor_j(V^{es, *}k(0), k(0)) \neq 0, \] where the $V$-structure on $k(0)$ corresponds to the canonical Cartier module structure on $\omega_A$. We conclude again by  \cite[Example 3.4.1]{Baudin_Positive_characteristic_generic_vanishing_theory}.
\end{proof}

Our final goal in this section is to show certain compatibilities related with the constructions from \autoref{restriction_of_a_V_module_to_subschemes} and \autoref{rem:upper_shriek_and_pullback_are_equal_for_isogeny}. However, we face an issue: there does not seem to be a canonical choice for an isomorphism of functors $V^{\flat} \cong V^*$ (or $F$ or any isogeny). The way we obtained it in the proof of \autoref{restriction_of_a_V_module_to_subschemes} was through the choice of an isomorphism $\cO_A \to \omega_A$ of coherent sheaves.

The reason why this is problematic is that given a morphism $f \colon A \to B$ of abelian varieties and a $V$--module $\cN$ on $A$, the induced $V$--module $f_*\cN$ on $B$ might have different structural morphisms depending on our choice of trivializations of $\omega_A$ and $\omega_B$. In order to show certain compatibilities, we will need to be extra careful with this.

First of all, note that two different isomorphisms $V^{\flat} \cong V^*$ differ by an isomorphism $V^* \to V^*$. Let us study these.

\begin{lem}\label{natural_transfo_pullback}
	Let $X$ be a Noetherian $\bF_p$--scheme, and let $F^{s, *} \to F^{s, *}$ be a natural transformation on $\Coh_X$. Then it is given by multiplication by an element of $H^0(X, \cO_X)$. The same holds if we replace $F^s$ by $V^s$ on $X = \bighat{A}$.
\end{lem}
\begin{proof}
	By flatness of $V^s$ and points \autoref{itm:equiv_cat} and \autoref{itm:behaviour_pushforwards_and_pullbacks} of \autoref{properties_Fourier_Mukai_transform}, it is enough to prove the first statement. Since $F^{s, *}$ preserves colimits, we automatically obtain an induced natural transformation $F^{s, *} \to F^{s, *}$ on $\QCoh_X$. If $j \colon U \inj X$ is an open immersion and $\cM \in \QCoh_U$, we then obtain a map $F^{s, *}(j_*\cM) \to F^{s, *}(j_*\cM)$, and therefore by restriction a map $F^{s, *}\cM \to F^{s, *}\cM$. We thus obtain such a natural transformation for all open $U \inc X$ as well, compatible with restrictions to open subsets.
	
	Since there is a canonical isomorphism $F^{s, *}\cO_X \cong \cO_X$, the map $F^{s, *}\cO_X \to F^{s, *}\cO_X$ corresponds to an element $\lambda_X \in H^0(X, \cO_X)$. The same is also true for any open $U \inc X$, and by naturality and compatibility, $\lambda_X|_U = \lambda_U$ (we will denote this element by $\lambda$). We will show that $F^{s, *} \to F^{s, *}$ is given by multiplication by $\lambda$. This is a local question, and we may therefore assume that $X = \Spec R$ is affine. Note that if $\cM \surj \cN$ is a surjection and we already know that $F^{s, *}\cM \to F^{s, *}\cM$ is given by multiplication by $\lambda$, then the same holds for $F^{s, *}\cN \to F^{s, *}\cN$. Since we know the result for $\cO_X$, we know it for $\cO_X^{\oplus n}$ for all $n \geq 1$ (a natural transformation must preserve splittings by definition), and therefore we know it for any globally generated module by the previous sentence. Since $X$ is affine, the proof is complete.
\end{proof}

\begin{lem}\label{multiplication_by_cst_does_not_change_Cartier_or_V}
	Let $(\cM, \theta)$ be a Cartier module on a Noetherian scheme $X$ over $k$. Given $\lambda \in k^{\times}$, consider the Cartier module $(\cM, \theta')$ where $\theta'(m) \coloneqq \theta(\lambda m)$. Then there is a natural isomorphism $(\cM, \theta) \to (\cM, \theta')$.
	
	The same holds for $V$--modules on $\bighat{A}$.
\end{lem}
\begin{proof}
	If $\mu$ denotes a $(p^s - 1)$'th root of $\lambda$ (recall that $k$ is separably closed), then multiplication by $\mu$ is an isomorphism of Cartier modules $(\cM, \theta) \to (\cM, \theta')$. For a $V$--module $\cN$, apply the same proof as above to $\FM_{\bighat{A}}(\cN)$ and then apply $\FM_A$.
\end{proof}

Combining \autoref{natural_transfo_pullback} and \autoref{multiplication_by_cst_does_not_change_Cartier_or_V}, we see that the functors we defined on Cartier modules and $V$--modules in \autoref{restriction_of_a_V_module_to_subschemes} and \autoref{rem:upper_shriek_and_pullback_are_equal_for_isogeny} do not depend on any choice. Let us use it to our advantage to obtain certain compatibilities which we will need later.

\begin{lem}\label{annyoing_compatibilities}
	Let $\pi \colon A \to B$ be a morphism of abelian varieties with dual $\bighat{\pi} \colon \bighat{B} \to \bighat{A}$, and let $\cN$ be a $V$--module on $\bighat{B}$. Then there is an isomorphism of Cartier modules \[ \pi^*\cH^0\FM_{\bighat{B}}(\cN) \cong \cH^0\FM_{\bighat{A}}(\bighat{\pi}_*\cN). \]
\end{lem}
\begin{proof}
	By \autoref{natural_transfo_pullback} and \autoref{multiplication_by_cst_does_not_change_Cartier_or_V}, we see that it is enough to show this result by fixing a specific isomorphism $F^{\flat} \cong F^*$ on $A$ and $B$. Let us choose one that makes our problem easier.
	
	We first fix an isomorphism $V^{\flat} \cong V^*$ on $\bighat{A}$, so that this gives $V^*$ the structure of a right adjoint of $V_*$. Applying $\FM_{\bighat{A}}$ and \autoref{properties_Fourier_Mukai_transform}.\autoref{itm:behaviour_pushforwards_and_pullbacks}, this gives $F_*$ the structure of a left adjoint of $F^*$, and therefore an isomorphism $F^{\flat} \cong F^*$ on $A$. The benefit of this construction is the following: if $\cN' \to V^{s, *}\cN'$ denotes the structural morphism of a $V$--module $\cN'$ and we apply $\FM_{\bighat{A}}$ to obtain a morphism $F^s_*\FM_{\bighat{A}}(\cN') \to \FM_{\bighat{A}}(\cN')$, then the associated morphism $\FM_{\bighat{A}}(\cN') \to F^{s, *}\FM_{\bighat{A}}(\cN')$ is \emph{precisely} the image by $\FM_{\bighat{A}}$ of the associated map $V^s_*\cN' \to \cN'$ (under the isomorphism from \autoref{properties_Fourier_Mukai_transform}.\autoref{itm:behaviour_pushforwards_and_pullbacks}). We do the same procedure on $B$ and $\bighat{B}$.
	
	It now follows from our construction above, together with naturality and functoriality of the isomorphism from \emph{loc. cit.}, that the diagram \[ \begin{tikzcd}
		L\pi^*\FM_{\bighat{B}}(\cN) \arrow[d] \arrow[rr, "\cong"] &  & \FM_A(R\bighat{\pi}_*\cN) \arrow[d] \\
		{F^{s, *}L\pi^*\FM_{\bighat{B}}(\cN)} \arrow[rr, "F^{s, *}(\cong)"] &  & {F^{s, *}\FM_A(R\bighat{\pi}_*\cN)}
	\end{tikzcd} \] commutes. Applying $\cH^0$ and \autoref{properties_Fourier_Mukai_transform}.\autoref{itm:support} concludes the proof.
\end{proof}

\begin{lem}\label{proj_formula_Cartier}
	Let $\pi \colon A \to B$ be a fibration of relative dimension $j$ between abelian varieties, and let $\cM$ be a Cartier module on $B$. Then there is an isomorphism of Cartier modules $R^j\pi_*(\pi^*\cM) \cong \cM$.
\end{lem}
\begin{proof}
	Throughout this proof, we let $F_A$ (resp. $F_B$) denote the Frobenius endomorphism on $A$ (resp. $B$). Since $R^j\pi_*\cO_A \cong \cO_B$ (we fix such a choice of isomorphism), we obtain by the projection formula that for all coherent sheaves $\cP$ on $B$, there is an isomorphism $R^j\pi_*(\pi^*\cP) \cong \cP$. Our objective is to make this Cartier equivariant (whenever $\cP$ is a Cartier module). By \autoref{natural_transfo_pullback} and \autoref{multiplication_by_cst_does_not_change_Cartier_or_V}, we may choose whichever isomorphisms $F^{\flat} \cong F^*$ on $A$ and $B$ to prove the statement. 
	
	Fix an isomorphism $F_A^{\flat} \cong F_A^*$, giving us a natural transformation $\lambda_A \colon F^s_{A, *}F_A^{s, *} \to \id$ by adjunction. Given a coherent sheaf $\cP$ on $B$, we therefore have a map $F^s_{A, *}F^{s, *}_A(\pi^*\cP) \to \pi^*\cP$. Since $F^s_{A, *}F^{s, *}_A(\pi^*\cP) = F^s_{A, *}\pi^*F^{s, *}_B(\cP)$, we can therefore apply $R^j\pi_*$ and the projection formula to obtain a natural transformation $\lambda_B \colon F^s_{B, *}F^{s, *}_B(\cP) \to \cP$. By adjunction, this gives a natural transformation $F^{s, *}_B \to F^{s, \flat}_B$. 
	
	Let us show that it is an isomorphism of functors. Since $F^{s, \flat}_B \cong F^{s, *}_B$, we know by \autoref{natural_transfo_pullback} that this natural transformation is either an isomorphism or zero. We will show that it is never zero (except on the zero sheaf). By adjunction, this is equivalent to showing that for all coherent $\cP \neq 0$, the map $F^s_{B, *}F^{s, *}_B(\cP) \to \cP$ constructed before is non--zero. We will show that this map is actually surjective. Since $R^j\pi_*$ preserves surjections (all fibers have dimension $j$, so $R^{> j}\pi_* = 0$), it is enough to show by construction that the natural transformation $F^{s}_{A, *}F^{s, \flat}_A \to \id$ is surjective. This amounts to showing that for any coherent module $R$ on a regular local ring and $m \in M$, there exists $f \colon F^s_*R \to M$ such that $f(1) = m$. This is equivalent to $R$ being $F$--pure, so this natural transformation is indeed surjective. We have therefore proven that $F^{s, *}_B \to F^{s, \flat}_B$ is an isomorphism.
	
	Let us now prove the statement. Consider the map $\theta_{\cM} \colon \cM \to F_B^{s, *}\cM$ associated to $F^s_*\cM \to \cM$ via the isomorphism $F^{s, *}_B \to F^{s, \flat}_B$ constructed above. We then pullback by $\pi$ to obtain a map $\theta_{\pi^*\cM} \colon \pi^*\cM \to \pi^*F_B^{s, *}\cM \cong F_A^{s, *}(\pi^*\cM)$. We then obtain by definition of an adjunction that the Cartier module structure on $\pi^*\cM$ is given by the composition 
	\[ \begin{tikzcd}
		{F^s_{A, *}(\pi^*\cM)} \arrow[rr, "F^s_{A, *}(\theta_{\pi^*\cM})"] &  & {F^s_{A, *}F^{s, *}_A(\pi^*\cM)} \arrow[rr, "\lambda_A"] &  & \pi^*\cM.
	\end{tikzcd} \]
	Applying $R^j\pi_*$, we obtain a diagram 
	\[ \begin{tikzcd}
		{R^j\pi_*F^s_{A, *}(\pi^*\cM)} \arrow[rrr, "R^j\pi_*F^s_{A, *}(\theta_{\pi^*\cM})"] \arrow[d, "\cong"'] & & & {R^j\pi_*F^s_{A, *}F^{s, *}_A(\pi^*\cM)} \arrow[rr, "R^j\pi_*(\lambda_A)"] \arrow[d, "\cong"] &  & R^j\pi_*\pi^*\cM \arrow[d, "\cong"] \\
		{F^{s}_{B, *}\cM} \arrow[rrr, "F^s_{B, *}(\theta_{\cM})"]                                 & &  & {F^s_{B, *}F^{s, *}_B(\cM)} \arrow[rr, "\lambda_B"]                                 &  & \cM.                                
	\end{tikzcd} \] The left square commutes by naturality of the projection formula, while the right square commutes by definition of $\lambda_B$.
\end{proof}

\section{The main theorems}

Here we will show all our main results (except the very last part of \autoref{intro:thm_A} about $P_4$). Throughout this section, fix a normal proper variety $X$, and let $a \colon X \to A$ denote its Albanese morphism. We also set $d = \dim(X)$ and $g = \dim(A)$. Recall that we work over an algebraically closed field $k$.

\subsection{Separability}

Let us generalize \cite[Proposition 1.4]{Hacon_Patakfalvi_Zhang_Bir_char_of_AVs}.
\begin{lem}\label{separability}
	Let $\Delta \geq 0$ be a $\bZ_{(p)}$--divisor on $X$, and let $m > 0$ such that $m\Delta$ is integral. Assume that we are in either of the following situations:
	
	\begin{enumerate}
		\item $a$ is generically finite, $\Delta = 0$ and $P_p(X) \leq p - 1$;
		\item $m = 1$, $h^0_{ss}(X, \cO_X(K_X + \Delta)) \neq 0$ and $P_p(X, \Delta) \leq p - 1$;
		\item $\kappa(K_X + \Delta) = 0$ and there exists a Cartier structure on $\cO_X(m(K_X + \Delta))$ for which $a_*\cO_X(m(K_X + \Delta)) \not\sim_C 0$.
	\end{enumerate}
	
	Then $a \colon X \to A$ is separable.
\end{lem}
\begin{proof}
	Assume by contradiction that $a$ is not separable. Going through the exact same lines as in the proof of \cite[Proposition 1.4]{Hacon_Patakfalvi_Zhang_Bir_char_of_AVs}, we obtain the existence of a finite purely inseparable morphism $f \colon X \to Y$ over $A$ of height one such that 
	\begin{equation}\label{equation_insep}
		pK_X \sim f^*K_Y + (p - 1)E
	\end{equation}  for some divisor $E$ on $X$ such that $h^0(X, \cO_X(E)) \geq 2$ (see \cite[Lemma 4.2]{Zhang_Abundance_for_non_uniruled_3_folds_with_non_trivial_albanese_in_pos_char}). In particular, $h^0(X, \cO_X(pE)) \geq p$. Let us now finish the proof in each case.
	\begin{enumerate}
		\item In this case, $Y \to A$ is also generically finite, meaning that $h^0(Y, \omega_Y) \geq 1$ by \cite[Theorem 4.3]{Baudin_Positive_characteristic_generic_vanishing_theory}. Thanks to \autoref{equation_insep}, we obtain a contradiction with the assumption that $P_p(X) \leq p - 1$.
		\item Since $f$ has height one, there exists a commutative diagram \[ \begin{tikzcd}
			X \arrow[r, "f"] \arrow[rr, "F"', bend right] & Y \arrow[r, "g"] & X \arrow[r, "f"] & Y.
		\end{tikzcd} \]
		Given that $h^0_{ss}(X, \cO_X(K_X + \Delta)) \neq 0$ and that $g^{\flat}\cO_X(K_X + \Delta) \cong \cO_Y(K_Y + g^*\Delta)$ (this holds on the regular locus, and one readily verifies that $g^{\flat}$ preserves reflexivity), we obtain by \autoref{universal_homeo_induce_equivalence_of_Cartier_crystals} that $h^0_{ss}(Y, \cO_Y(K_Y + g^*\Delta)) \neq 0$, so in particular $K_Y + g^*\Delta \geq 0$. Since \[ p(K_X + \Delta) \sim f^*(K_Y + g^*\Delta) + (p - 1)E  \] by \autoref{equation_insep}, we again obtain a contradiction since $P_p(X, \Delta) \leq p - 1$.
		\item We use the same notations as in the previous point. Write $m(K_X + \Delta) = K_X + D$. Then again by \autoref{universal_homeo_induce_equivalence_of_Cartier_crystals}, there exists a Cartier module structure on $\cO_Y(K_Y + g^*D)$ such that $a_*g_*\cO_Y(K_Y + g^*D) \not\sim_C 0$. By \autoref{generic vanishing}.\autoref{itm:gv_p_closed} and \autoref{exists_V^i_with_codim_i}, there exists a (prime--to--$p$) torsion point in $V^0_{\injj}(a_*g_*\cO_Y(K_Y + g^*D))$, so in particular there exists $n > 0$ such that $n(K_Y + g^*D)$ is effective. We obtain from \autoref{equation_insep} that \[ nmp(K_X + \Delta) \sim f^*(n(K_Y + g^*D)) + n(p - 1)E, \] contradicting that $\kappa(K_X + \Delta) = 0$. \qedhere
	\end{enumerate}
\end{proof}

\subsection{A positive characteristic analogue of Pareschi's argument}

Pareschi proved in \cite{Pareschi_Basic_results_on_irr_vars_via_FM_methods} that if $X$ is a smooth projective complex variety such that $V^0(a_*\omega_X)$ has positive dimension, then there exists a positive--dimensional closed subset $Z \inc V^0(a_*\omega_X)$ such that also $-Z \inc V^0(a_*\omega_X)$ (see \cite{Chen_Jiang_Positivity_in_varieties_of_maximal_Albanese_dimension} for an improvement). Our objective is to present a positive characteristic version of this statement and argument:

\begin{prop}\label{Pareschi_char_p}
	Assume that $a \colon X \to A$ is finite and that $A$ has no supersingular factor. If $\dim(V^0_{\injj}(a_*\tau(\omega_X))) > 0$, then there exists a positive--dimensional closed subset $Z \inc V^0_{\injj}(a_*\tau(\omega_X))$ such that also $-Z \inc V^0_{\injj}(a_*\tau(\omega_X))$.
\end{prop}

\begin{rem}\label{remark_test_module}
	\begin{enumerate}
		\item In \autoref{more_general_Pareschi}, we will generalize the above statement.
		\item\label{rem_test} We recall that $\tau(\omega_X) \inc \omega_X$ denotes the test submodule of $\omega_X$, namely the minimal non--zero sub--Cartier module of $\omega_X$ (see \cite[End of Section 4]{Smith_Test_ideals_in_local_rings}). A direct way to see its (global) existence is by using that Cartier crystals are Artinian (\cite{Blickle_Bockle_Cartier_modules_finiteness_results}) and \cite[Lemma 13.1]{Gabber_notes_on_some_t_structures} to go from crystals to modules. Note that whenever $X$ is regular, we have that $\tau(\omega_X) = \omega_X$. In particular, $\tau(\omega_X)|_{X_{\reg}} = \omega_{X_{\reg}}$.
	\end{enumerate}
	
\end{rem}

The proof will be split into several parts. Throughout, we assume that $X$ satisfies the hypotheses of \autoref{Pareschi_char_p}. If $V^0_{\injj}(a_*\tau(\omega_X)) = \bighat{A}$, then there is nothing to show, so we may and will assume that $V^0_{\injj}(a_*\tau(\omega_X)) \neq \bighat{A}$. We also let $\cN$ denote the injective $V$--module associated to $a_*\tau(\omega_X)$.

\begin{lem}\label{fibered_ab_vars}
	The image of $a \colon X \to A$ is fibered by abelian subvarieties.
\end{lem}
\begin{proof}
	By assumption and \autoref{generic vanishing}.\autoref{itm:gv_p_closed}, there exists an integer $s > 0$, a line bundle $\beta \in \bighat{A}$ such that $(p^s - 1)(\beta) = 0$ and a strict abelian subvariety $\bighat{B} \inc \bighat{A}$ such that $\beta + \bighat{B}$ is a positive--dimensional irreducible component of $V^0_{\injj}(a_*\tau(\omega_X))$ (so $-\beta + \bighat{B}$ is a component of $\Supp(\cN)$). We may see $a_*\tau(\omega_X)$ as an $s$--Cartier module, and therefore consider the Cartier module $a_*\tau(\omega_X) \otimes \beta$ (see \autoref{pullback_of_Cartier_modules}.\autoref{itm:tensor_product_of_Cartier_module_and_unit_F_module}), with associated injective $V$--module $\cN' \coloneqq T_{-\beta}^*\cN$ (see \autoref{properties_Fourier_Mukai_transform}.\autoref{itm:behaviour_under_translation}). Let $\pi \colon A \surj B$ denote the fibration dual to $\bighat{B} \inc \bighat{A}$.
	
	%	
	%	
	%	We can now explain how the proof will proceed. Since the fibers of $\pi$ are all abelian subvarieties of dimension $j$, it is enough to show that the fibers of $a(X) \to \pi(a(X)) \inc B$ are also of dimension $j$ (because then they will agree with the fibers of $\pi$). The way we will show this statement is by proving that $R^j\pi_*(a_*\omega \otimes \beta)$ is supported everywhere on $\pi(a(X))$.
	
	%	By construction, $\bighat{B}$ is an irreducible component of $\Supp(\cN')$. Consider the sub--$V$--module $i_{\bighat{B}, *}i_{\bighat{B}}^{\flat}(\cN')$. By \JB{ref}, we know that \[ L^ji_{\bighat{B}}^*(i_{\bighat{B}, *}i_{\bighat{B}}^{\flat}(\cN')) \sim_V i_{\bighat{B}}^{\flat}(\cN'), \] so we obtain by \JB{ref} that \[ R^j\pi_*\FM_{\bighat{A}}(i_{\bighat{B}, *}i_{\bighat{B}}^{\flat}(\cN')) \sim_C \FM_{\bighat{B}}(i_{\bighat{B}^{\flat}}(\cN')). \]
	
	By construction, there is an injection $i_{\bighat{B}, *}i_{\bighat{B}}^{\flat}(\cN') \inj \cN'$, so by \autoref{equivalence_V_crystals_and_Cartier_crystals} there is a non--zero surjection of Cartier crystals \[ a_*\tau(\omega_X) \otimes \beta \dashrightarrow \cH^0\FM_{\bighat{A}}(i_{\bighat{B}, *}i_{\bighat{B}}^{\flat}(\cN')) \expl{\cong}{\autoref{annyoing_compatibilities}} \pi^*\cH^0\FM_{\bighat{B}}(i_{\bighat{B}}^{\flat}\cN'). \]  Thus, \[ a(X) \supseteq \Supp_{\crys}(\pi^*\cH^0\FM_{\bighat{B}}(i_{\bighat{B}}^{\flat}\cN')) \expl{=}{\autoref{support_of_pullback}} \pi^{-1}(\Supp_{\crys}(\cH^0\FM_{\bighat{B}}(i_{\bighat{B}}^{\flat}\cN'))). \] To conclude the proof, it is then enough to show that $\Supp_{\crys}(\cH^0\FM_{\bighat{A}}(i_{\bighat{B}, *}i_{\bighat{B}}^{\flat}(\cN'))) = a(X)$. Since we have a non--zero morphism of crystals $a_*\tau(\omega_X) \otimes \beta \dashrightarrow \cH^0\FM_{\bighat{A}}(i_{\bighat{B}, *}i_{\bighat{B}}^{\flat}(\cN'))$, there is also a non--zero morphism of crystals $a_*\tau(\omega_X) \dashrightarrow \cH^0\FM_{\bighat{A}}(i_{\bighat{B}, *}i_{\bighat{B}}^{\flat}(\cN')) \otimes \beta^{-1}$. By adjunction (recall that $a$ is finite), we obtain a non--zero morphism of crystals \[ \tau(\omega_X) \dashrightarrow a^{\flat}(\cH^0\FM_{\bighat{A}}(i_{\bighat{B}, *}i_{\bighat{B}}^{\flat}(\cN')) \otimes \beta^{-1}). \] Since $\tau(\omega_X)$ is simple by definition, this morphism must be injective. In particular, we obtain that $\Supp_{\crys}(a^{\flat}(\cH^0\FM_{\bighat{A}}(i_{\bighat{B}, *}i_{\bighat{B}}^{\flat}(\cN')) \otimes \beta^{-1})) = X$, so $\Supp_{\crys}(\cH^0\FM_{\bighat{A}}(i_{\bighat{B}, *}i_{\bighat{B}}^{\flat}(\cN')) \otimes \beta^{-1}) = a(X)$. \qedhere
	
	%	
	%	
	%	
	%	
	%	 Since all fibers of $A \to B$ are $j$--dimensional, we obtain a surjection of Cartier crystals \[ R^j\pi_*(a_*\omega \otimes \beta) \surj R^j\pi_*\FM_{\bighat{A}}(i_{\bighat{B}, *}i_{\bighat{B}}^{\flat}(\cN')) \sim_C \FM_{\bighat{B}}(i_{\bighat{B}^{\flat}}(\cN')). \] Thus, it is enough to show that $\Supp_{\crys}(\FM_{\bighat{B}}(i_{\bighat{B}^{\flat}}(\cN'))) = \pi(a(X))$. By pulling back, it is enough to show that $\Supp_{\crys}(\pi^*\FM_{\bighat{B}}(i_{\bighat{B}^{\flat}}(\cN'))) = a(X)$.
	%	
	%	Note that $\cH^0\FM_A(\cN'') \sim_C \pi^*\cH^0(\FM_A(i_{\bighat{B}}(\cN)))$ by \JB{ref}. 
\end{proof}

\begin{rem}
	Note that the proof above shows something more general. Namely it shows that if $X$ satisfies the assumptions of \autoref{Pareschi_char_p} and $\cN$ is not torsion--free, then $a(X)$ is fibered by abelian subvarieties (although we may have $V^0_{\injj}(\cN) = \bighat{A}$ in this case). 
\end{rem}

We will keep the notations of the proof of \autoref{fibered_ab_vars}. Let $j > 0$ denote the relative dimension of $\pi \colon A \surj B$.

\begin{lem}\label{higher_pushforward_supported_everywhere}
	The Cartier crystal $R^j\pi_*(a_*\tau(\omega_X) \otimes \beta)$ is supported everywhere on $\pi(a(X))$, and $V^0_{\injj}(R^j\pi_*(a_*\tau(\omega_X) \otimes \beta)) = \bighat{B}$. 
\end{lem}
\begin{proof} 
	Recall that by the proof of \autoref{fibered_ab_vars}, we have a surjection of Cartier crystals \[ a_*\tau(\omega_X) \otimes \beta \dashrightarrow \pi^*\cH^0\FM_{\bighat{B}}(i_{\bighat{B}}^{\flat}\cN'). \] Since fibers of $\pi$ are all of dimension $j$, there is no $R^{j + 1}\pi_*$ and we therefore obtain a surjection \[ R^j\pi_*(a_*\tau(\omega_X) \otimes \beta) \dashrightarrow R^j\pi_*(\pi^*\cH^0\FM_{\bighat{B}}(i_{\bighat{B}}^{\flat}\cN')) \expl{\sim_C}{\autoref{proj_formula_Cartier}} \cH^0\FM_{\bighat{B}}(i_{\bighat{B}}^{\flat}\cN').  \] In particular, the injective $V$--module $\cR$ associated to $R^j\pi_*(a_*\tau(\omega_X) \otimes \beta)$ contains $i_{\bighat{B}}^{\flat}\cN'$ as a sub--$V$--crystal. By the $V$--module analogue of \cite[Proposition 3.4.6]{Bockle_Pink_Cohomological_Theory_of_crystals_over_function_fields}, this shows the existence of a morphism of $V$--modules $i^{\flat}_{\bighat{B}}(\cN') \to V^{es, *}\cR$ for some $e \geq 0$, which is injective at the level of crystals (hence its kernel is nilpotent). Since $i_{\bighat{B}}^{\flat}\cN'$ is an injective $V$--module, this morphism is actually injective. Given that $i^{\flat}_{\bighat{B}}(\cN')$ is supported everywhere on $\bighat{B}$, we deduce that so is $\cR$, i.e. $V^0_{\injj}(R^j\pi_*(a_*\tau(\omega_X) \otimes \beta)) = \bighat{B}$. The surjection above also gives that \[ \pi(a(X)) \cni \Supp_{\crys}(R^j\pi_*(a_*\tau(\omega_X) \otimes \beta)) \cni \Supp_{\crys}(\cH^0\FM_{\bighat{B}}(i_{\bighat{B}}^{\flat}\cN')). \] Recall also from the proof of \autoref{fibered_ab_vars} that \[ \pi^{-1}(\Supp_{\crys}(\cH^0\FM_{\bighat{B}}(i_{\bighat{B}}^{\flat}\cN'))) = a(X), \] so \[ \Supp_{\crys}(\cH^0\FM_{\bighat{B}}(i_{\bighat{B}}^{\flat}\cN')) = \pi(a(X)). \qedhere \] 
\end{proof}

Note that if $\beta \in \bighat{B}$, then $\bighat{B}$ is a positive--dimensional component of $V^0_{\injj}(a_*\tau(\omega_X))$. Since $\bighat{B} = -\bighat{B}$, \autoref{Pareschi_char_p} would follow automatically. Therefore, we will assume from now on that $\beta \notin \bighat{B}$.

\begin{lem}\label{twist_by_beta_inverse_also_good}
	We have that $R^j\pi_*(a_*\tau(\omega_X) \otimes \beta^{-1}) \not\sim_C 0$, and that $V^0_{\injj}(R^j\pi_*(a_*\tau(\omega_X) \otimes \beta^{-1}))$ has no isolated point (so it is positive--dimensional).
\end{lem}
\begin{proof}
	Let \[ \begin{tikzcd}
		X \arrow[r, "f"] & Y \arrow[r, "b"] & \pi(a(X))
	\end{tikzcd} \] denote the Stein factorization of $\pi \circ a \colon X \to \pi(a(X)) \inc B$, and let $X_{\eta}$ denote the generic fiber of $f \colon X \to Y$. By \autoref{higher_pushforward_supported_everywhere}, we know that \[ R^j\pi_*(a_*\tau(\omega_X) \otimes \beta) \expl{\cong}{$a$ is finite} R^j(\pi \circ a)_*(\tau(\omega_X) \otimes a^*\beta) \cong b_*R^jf_*(\tau(\omega_X) \otimes a^*\beta) \] is supported everywhere on $\pi(a(X))$ as a crystal, so \[ H^j(X_{\eta}, (\tau(\omega_X) \otimes \beta)|_{X_{\eta}}) \not\sim_C 0. \] Note that by \autoref{fibered_ab_vars}, we know that $j = \dim(X_{\eta})$. Since $X_{\eta}$ is normal (hence regular in codimension $1$), we know by \autoref{remark_test_module}.\autoref{rem_test} that the cokernel of $(\tau(\omega_X) \otimes \beta)|_{X_{\eta}} \inc \omega_{X_{\eta}} \otimes \beta|_{X_{\eta}}$ is supported in codimension at least two, so \[ H^j(X_{\eta}, \omega_{X_{\eta}} \otimes \beta|_{X_{\eta}}) \not\sim_C 0. \] By Serre duality and \cite[Proposition 2.4]{Patakfalvi_Zdanowicz_Ordinary_varieties_with_trivial_canonical_bundle_are_not_uniruled}, we deduce that \[ H^0(X_{\eta}, \beta^{-1}|_{X_{\eta}}) \not\sim_F 0, \] which implies that there exists a non--zero section in global section of $\beta^{-1}_{X_{\eta}}$ fixed by the Frobenius structure. This global section induces a non--zero morphism of \emph{Frobenius modules} $\cO_{X_{\eta}} \to \beta^{-1}|_{X_{\eta}}$, which is automatically an isomorphism since $\beta^{-1}|_{X_{\eta}} \in \Pic^0(X_{\eta}$). The same then holds for $\beta|_{X_{\eta}}$, so reversing the above argument shows that \[ R^j\pi_*(a_*\tau(\omega_X) \otimes \beta^{-1}) \not\sim_C 0. \]

	Assume now by contradiction that $V^0_{\injj}(R^j\pi_*(a_*\tau(\omega_X) \otimes \beta^{-1}))$ has an isolated point $\gamma \in \bighat{B}$. It must then be in $V^{g - j}_{\injj}(R^j\pi_*(a_*\tau(\omega_X) \otimes \beta^{-1}))$ (recall that $g - j = \dim(B))$. We then obtain by \autoref{generic vanishing}.\autoref{van_and_non_van} that $\gamma \in W^{g - j}_F(R^j\pi_*(a_*\tau(\omega_X) \otimes \beta^{-1}))$, so in particular  \[ H^{g - j}(B, R^j\pi_*(a_*\tau(\omega_X) \otimes \beta^{-1}) \otimes \gamma^{p^{es}}) \neq 0 \] for $e \gg 0$. Since
	\begin{align*}
		H^{g - j}(B, R^j\pi_*(a_*\tau(\omega_X) \otimes \beta^{-1}) \otimes \gamma^{p^{es}}) & \cong H^g(A, a_*\tau(\omega_X) \otimes \beta^{-1} \otimes \pi^*\gamma^{p^{es}}) \\ & \cong H^g(X, \tau(\omega_X) \otimes a^*(\beta^{-1} \otimes \pi^*\gamma^{p^{es}})),
	\end{align*}
	we automatically deduce that $g = d$ and that, more importantly, $a^*(\beta^{-1} \otimes \pi^*\gamma^{p^{es}}) \cong \cO_X$. By \autoref{existence_of_albanese}, this shows that $\beta \cong \pi^*\gamma^{p^{es}} \in \bighat{B}$, contradicting our assumptions.
\end{proof}

\begin{lem}\label{final_bit_of_the_proof_of_Pareschi}
	Let $Z' \inc V^0_{\injj}(R^j\pi_*(a_*\tau(\omega_X) \otimes \beta^{-1})) \inc \bighat{B}$ be an irreducible component of maximal dimension. Then $-\beta + p^{es}(Z') \inc V^0_{\injj}(a_*\tau(\omega_X))$ for some $e > 0$. In particular, \autoref{Pareschi_char_p} holds.
\end{lem}
\begin{proof}
	The part after ``In particular'' is really an immediate consequence, since $-\beta + p^{es}(Z') \inc -\beta + p^{es}(\bighat{B}) = -\beta + \bighat{B} = -(\beta + \bighat{B})$, while $\beta + \bighat{B} \inc V^0_{\injj}(a_*\tau(\omega_X))$ by construction.
	
	Let us therefore show the first statement. We may base change and assume that $k$ is uncountable. Let $r \geq 0$ denote the codimension of $Z'$ in $\bighat{B}$. By \autoref{very_gen_pts_in_V0_implies_in_V0}, it is enough to check that for very general $\alpha \in Z'$, \[ H^{j + r}_{ss}(A, a_*\tau(\omega_X) \otimes \beta^{-1} \otimes \pi^*\alpha) \neq 0 \] (recall that $p^s(\beta) = \beta$). By \autoref{generic vanishing}.\autoref{van_and_non_van}, we know that for $\alpha \in Z'$ very general, we have \[ H^r_{ss}(B, R^j\pi_*(a_*\tau(\omega_X) \otimes \beta^{-1}) \otimes \alpha) \neq 0. \] Hence, it is enough to show that for $\alpha \in Z'$ very general, \[ H^{j + r}_{ss}(A, a_*\tau(\omega_X) \otimes \beta^{-1} \otimes \pi^*\alpha) \cong H^r_{ss}(B, R^j\pi_*(a_*\tau(\omega_X) \otimes \beta^{-1}) \otimes \alpha). \] By a Leray spectral sequence argument and the fact that $R^{>j}\pi_* = 0$ (all fibers of $\pi$ are $j$--dimensional), it is enough to show that for all $i < j$, $l > r$ and $\alpha \in Z'$ very general, \[ H^l_{ss}(B, R^i\pi_*(a_*\tau(\omega_X) \otimes \beta^{-1}) \otimes \alpha) = 0. \] This follows immediately from the fact that $V^l_{\injj}(R^i\pi_*(a_*\tau(\omega_X) \otimes \beta^{-1})) \inc \bighat{B}$ has codimension at least $l > r$ (see \autoref{generic vanishing}.\autoref{itm:gv_codim}) and \autoref{generic vanishing}.\autoref{van_and_non_van}.
\end{proof}

%\begin{proof}[End of the proof of \autoref{Pareschi_char_p}]
%	By \autoref{final_bit_of_the_proof_of_Pareschi}, we have proven overall that if $Z = \beta + \bighat{B}$ is a positive--dimensional irreducible component of $V^0_{\injj}$ with $p^s(\beta) = \beta$ and $\bighat{B}$ an abelian subvariety of $\bighat{A}$, then there exists an integer $e > 0$ and a closed subset $Z' \inc -Z$ such that $p^{es}(Z') \inc V^0_{\injj}$. Since $p^{es}(Z') \inc p^{es}(Z) = Z$, the proof is complete.
%\end{proof}

We are therefore done with the proof of \autoref{Pareschi_char_p}. As a side note, remark that one can generalize \autoref{fibered_ab_vars} as follows. 

\begin{prop}\label{more_general_Albanese_image}
	Assume that $a$ is generically finite, that $A$ has no supersingular factor and that $V^0_{\injj}(a_*\omega) \neq \bighat{A}$ for some non--zero Cartier module $\omega \inc \omega_X$. Then $a(X)$ is fibered by abelian varieties.
\end{prop} 
\begin{proof}
	Let \[ \begin{tikzcd}
		X \arrow[r, "f"] & Y \arrow[r, "b"] & A
	\end{tikzcd}\] denote the Stein factorization of $a$. Since $f_*\omega \inc f_*\omega_X \inc \omega_Y$ and $f_*\omega$ is non--nilpotent (it agrees with $f_*\omega_X$ generically by \autoref{remark_test_module}.\autoref{rem_test}), we know by definition $\tau(\omega_Y) \inc f_*\omega$. Hence, we also have that $b_*\tau(\omega_Y) \inc a_*\omega$, so $V^0_{\injj}(b_*\tau(\omega_Y)) \neq \bighat{A}$. We then conclude by \autoref{fibered_ab_vars}.
\end{proof}

\subsection{The case of finite cohomological support locus}

Throughout, we fix an integer $s > 0$. All our Cartier modules will be seen as $s$--Cartier modules. Our objective is to show the following result:

\begin{prop}\label{V0_zero_dim_implies_done}
	Assume that $a$ is generically finite, let $\omega \inc \omega_X$ be a sub--Cartier module, and assume that $V^0_{\injj}(a_*\omega)$ is zero--dimensional. Then $a$ is surjective, and the finite part of $a \colon X \to A$ is purely inseparable.
\end{prop}
\begin{proof}
	As in the proof of \autoref{more_general_Albanese_image}, we may assume that $a$ is finite. Let $\cN$ be the injective $V$--module associated to $a_*\omega$. \\
	
	\noindent \textbf{Step 1:} \emph{The Albanese map is surjective, and the closed subset $V^0_{\injj}(a_*\omega) = -\Supp(\cN) \inc \bighat{A}$ consists only of $p$--power torsion line bundles.} \\
	
	Let us start with proving that $\Supp(\cN)$ consists of $p$--power torsion points. Let $\beta \in \Supp(\cN)$. Given that $\cN$ is supported in dimension zero by assumption, there exists an injection $k(\beta) \inj \cN$ of coherent sheaves. In particular, $0 \neq \Tor_g(k(\beta), k(\beta)) \inj \Tor_g(\cN, k(\beta))$ so $-\beta \in V^g_{\injj}(a_*\omega)$. By \autoref{generic vanishing}.\autoref{van_and_non_van}, we deduce that \[ \lim H^g(A, F^{es}_*a_*\omega \otimes \beta^{-1}) \neq 0. \] However, \[ H^g(A, F^{es}_*a_*\omega \otimes \beta^{-1}) \cong H^g(A, a_*\omega \otimes \beta^{-p^{es}}) \expl{\cong}{$a$ is finite} H^g(X, \omega \otimes a^*\beta^{-p^{es}}) \: \expl{\cong}{as in the proof of \autoref{twist_by_beta_inverse_also_good}} \: H^0(X, a^*\beta^{-p^{es}}) \] so given that $a^*$ is injective (\autoref{existence_of_albanese}), we deduce that $p^{es}(\beta)= 0$ for $e \gg 0$. 
	
	Let us finish now prove the surjectivity of $a$. First, note that $a_*\omega \not\sim_C 0$, since $F^s_*(a_*\omega) \to a_*\omega$ is generically surjective on $a(X)$ ($a$ is generically finite). We therefore obtain that $V^0_{\injj}(a_*\omega) \neq \emptyset$, and hence the argument above shows that $H^g(A, F^{es}_*a_*\omega \otimes \beta^{-1}) \neq 0$ for some $\beta \in \bighat{A}$ and $e > 0$. In particular, the support of the coherent sheaf $F^{es}_*a_*\omega$ must be the whole $A$ (otherwise $H^g$ would vanish for dimensionsal reasons). This shows that $a$ is surjective. \\
	
%	
%	 Since $a_*\omega \sim_C \FM_{\bighat{A}}(\cN)$ by \autoref{equivalence_V_crystals_and_Cartier_crystals}, we deduce that $\FM_{\bighat{A}}(\cN)$ is also not nilpotent. Let $\cM \inc \FM_{\bighat{A}}(\cN)$ be the image of $F^{es}_*\FM_{\bighat{A}}(\cN) \to \FM_{\bighat{A}}(\cN)$ for $e \gg 0$ (these images stabilize by \cite[Lemma 13.1]{Gabber_notes_on_some_t_structures}). Since $\cM \neq 0$ and $\FM_{\bighat{A}}(\cN)$ is torsion--free, we automatically obtain that $\cM$ is supported everywhere on $A$. Given that the structural morphism of $\cM$ is surjective by construction, one readily sees that $\Supp_{\crys}(\cM) = A$. Since $a_*\omega \sim_C \FM_{\bighat{A}}(\cN) \sim_C \cM$, we deduce that $\Supp_{\crys}(a_*\omega) = A$. This shows in particular that $a_*\omega$ is supported everywhere as a coherent sheaf, so $a$ must be surjective.
%	
%	Let us now show that $a$ is surjective. Since $a_*As it turns out, the argument above show that 

	\noindent \textbf{Step 2:} \emph{There exists a unipotent vector bundle $\cV'$ on $A$ with the structure of a unit Cartier module such that $\cV' \sim_C a_*\omega$.} \\ 
	
	Consider the adjoint structural morphism $\psi \colon V^s_*\cN \to \cN$ (see \autoref{restriction_of_a_V_module_to_subschemes}), and let $\cN'$ denote the image of $\psi^{e}$ for $e \gg 0$ (these images stabilize since $\cN$ has finite length). By construction, $V^s_*\cN' \to \cN'$ is surjective. Given that $\cN'$ and $V^s_*\cN'$ have the same length, we automatically obtain that $V^s_*\cN' \to \cN'$ is an isomorphism. In particular, \[ \Supp(\cN') = \Supp(V^s_*\cN') = V^s(\Supp(\cN')), \] so given that $V$ is set--theoretically the same as multiplication by $p$ on $\bighat{A}$, we deduce by Step 1 that $\Supp(\cN') = \{0\}$. We then know by \autoref{properties_Fourier_Mukai_transform}.\autoref{itm:equiv_cat_unipotent} that $\cV' \coloneqq \FM_{\bighat{A}}(\cN')$ is a unipotent vector bundle. Furthermore, the fact that $V^s_*\cN' \to \cN'$ is an isomorphism implies that $\cV'$ is unit (use for example \autoref{filtration_for_V_modules_supp_at_0} or the same strategy as in the proof of \autoref{annyoing_compatibilities}). Finally, since the morphism $V^{es}_*\cN \to \cN$ factors through $\cN'$ for $e \gg 0$, we obtain that $\cN' \sim_V \cN$, whence $a_*\omega \sim_C \cV'$ by \autoref{equivalence_V_crystals_and_Cartier_crystals}. \\ 
	
	\noindent \textbf{Step 3:} \emph{Perform some étale base change to show that $\cV'$ has rank $1$. In other words, $a_*\omega \sim_C \omega_A$}. \\
	
	The idea is to use the same étale base change trick as in \cite{Hacon_Patakfalvi_Zhang_Bir_char_of_AVs}. Since $\cV'$ is unit, we obtain that $j \coloneqq \rank(\cV') = \rank_{\crys}(\cV') = \rank_{\crys}(a_*\omega)$. Our goal is to show that $j = 1$.
	 
	Since $\cN'$ is supported at $0$, there exists $N > 0$ such that $F^N_*\cN'$ is a $k(0)$--vector space, hence $F^N_*\cN' \cong k(0)^{\oplus j}$. In other words, $V^{N, *}(\cV') \cong \omega_A^{\oplus j}$ by \autoref{properties_Fourier_Mukai_transform}.\autoref{itm:behaviour_pushforwards_and_pullbacks} (note that here, $V$ denotes the Verschiebung of $A$ and not that of $\bighat{A}$). Write $V^N = \phi \circ \mu$, where $\mu \colon A \to B$ is purely inseparable and $\phi \colon B \to A$ is étale.
	
	Let us show that as coherent sheaves, $\phi^*\cV' \cong \omega_B^{\oplus j}$. Since $\mu$ is purely inseparable, we have a factorization \[ \begin{tikzcd}
		B \arrow[rr, "\ttilde{\mu}"] \arrow[rrrr, "F^M"', bend right] &  & A \arrow[rr, "\mu"] &  & B
	\end{tikzcd}  \] for some $M > 0$. Given that $F^* = F^{\flat}$ on an abelian variety (see \autoref{rem:upper_shriek_and_pullback_are_equal_for_isogeny}) and that $\cV'$ is unit, we have by definition an isomorphism $F^{M, *}\phi^*\cV' \cong \phi^*\cV'$ (up to making $M > 0$ bigger). Hence, \[ \phi^*\cV' \cong F^{M, *}\phi^*\cV' \cong \ttilde{\mu}^*\mu^*\phi^*\cV' \cong \ttilde{\mu}^*\omega_A^{\oplus j} \cong \omega_B^{\oplus j}. \]
	Now, consider the base change square \[
	\begin{tikzcd}
		Y \arrow[r, "\psi"] \arrow[d, "b"'] & X \arrow[d, "a"]              \\
		B \arrow[r, "\phi"']                     & A.               
	\end{tikzcd} 
	\] Since $\phi$ and $\psi$ are étale, we have $\phi^{\flat} = \phi^*$ and $\psi^{\flat} = \psi^*$ (see \stacksproj{0FWI}), so we can pullback Cartier modules (and crystals) by $\phi$ and $\psi$. Let $\omega' \coloneqq \psi^*\omega \inc \omega_Y$. By flat base change, we obtain that \[ \omega_B^{\oplus j} \sim_C \phi^*(a_*\omega) \cong b_*\omega'. \]
	
	By the same argument as in the proof of \cite[Theorem 3.2]{Hacon_Patakfalvi_Zhang_Bir_char_of_AVs}, we know that $Y$ is connected (and hence normal since $b$ is étale). We can therefore argue as in the proof of \autoref{twist_by_beta_inverse_also_good} and obtain that $h^d_{ss}(Y, \omega') = h^d_{ss}(Y, \omega_Y) = 1$. Let $\theta$ denote the induced unit Cartier module structure on $b_*\omega'$. Since $b$ is finite, we see that $h^d(Y, \omega') = h^d_{ss, \theta}(B, \omega_B^{\oplus j})$, so we are left to show that $h^d_{ss, \theta}(B, \omega_B^{\oplus j}) = j$ to finish this step. Given that $h^d(B, \omega_B^{\oplus j}) = j$ (recall that $d = g$ since $a$ is surjective by Step 1), it is enough to show that the induced map $H^d(\theta)$ is bijective. By definition of an adjunction, the morphism $\theta$ is given by the composition \[ \begin{tikzcd}
		F^s_*(\omega_B^{\oplus j}) \arrow[rr, "\theta^{\flat}"] &  & F^s_*F^{s, \flat}(\omega_B^{\oplus j}) \arrow[rr] &  & \omega_B^{\oplus j},
	\end{tikzcd} \] where the right map is the counit of the adjunction (in particular it does not depend on $\theta$). Given that the left map in an isomorphism, we have proven that the quantity $h^d_{ss, \theta}(B, \omega_B^{\oplus j})$ does not depend on the chosen \emph{unit} Cartier module stucture $\theta$ on $\omega_B^{\oplus j}$. Since this quantity is $j$ with the standard Cartier operator (apply Serre duality), we are done with this step. \\

%simply whether the right map is bijective after applying $H^d(B, -)$. In particular, it does not depend on the chosen unit Cartier structure $\theta$, so we have proven 
%	Let $C^{\flat} \colon \omega_B \to F^{\flat}\omega_B$ denote the isomorphism induced by the Cartier operator $C \colon F_*\omega_B \to \omega_B$. Then by definition of the Cartier operator and an adjunction, if we let $A_{\theta} \in \GL(g, k)$ be the matrix corresponding to the composition 
%	\[ \begin{tikzcd}
%		\omega_B^{\oplus j} \arrow[rr, "\theta^{\flat}"] &  & F^{\flat}(\omega_B^{\oplus j}) \arrow[rr, "((C^{\flat})^{-1})^{\oplus j}"] &  & \omega_B^{\oplus j},
%	\end{tikzcd} \]
%	then $\theta$ is given by the composition \[  \]

	\noindent \textbf{Step 4:} \emph{Conclude the proof.} \\
	
	\noindent We have that \[ 1 = \rank_{\crys}(\omega_A) \expl{=}{Step 4} \rank_{\crys}(a_*\omega) \expl{=}{\autoref{rank_of_pushforward}} [K(X) : K(A)]_{\sep}\rank_{\crys}(\omega), \] so the extension $K(X)/K(A)$ is purely inseparable.
\end{proof}

Let us also prove a variant of \autoref{V0_zero_dim_implies_done}.

\begin{prop}\label{V0_zero_dim_implies_done_ordinary_case}
	Let $\omega \inc \omega_X$ be a sub--Cartier module. Assume that $a$ is generically finite, that $V^0_{\injj}(a_*\omega) = \{0\}$ and that $A$ is ordinary. Then $a$ is birational.
\end{prop}
\begin{proof}
	By \autoref{V0_zero_dim_implies_done}, we already know that $a$ is surjective and purely inseparable. Let $\cN$ denote the injective $V$--module corresponding to $a_*\omega$, and let $\cN' \inc \cN$ be defined as in Step 2 of the proof of \autoref{V0_zero_dim_implies_done}.
	
	Assume that we already showed that $\cN' = \cN$. We then have a (non--zero) composition of morphisms \[ \begin{tikzcd}
		\FM_A(a_*\omega) \arrow[rr] &  & \cH^0\FM_A(a_*\omega) \arrow[rr] &  & \cN = \cN',
	\end{tikzcd} \] so applying $\FM_{\bighat{A}}$ and \autoref{properties_Fourier_Mukai_transform}.\autoref{itm:equiv_cat}, we obtain a non--zero morphism of Cartier \emph{modules} $\mu \colon \omega_A \to a_*\omega$ (recall that $\FM_{\bighat{A}}(\cN') \cong \omega_A$ by the proof of Step 3 above). Since $a_*\omega \sim_C \omega_A$ and this morphism is non--zero, it must be an isomorphism at the level of crystals. In particular, its cokernel is nilpotent. However, the structural morphism of $a_*\omega$ is generically surjective (since $X$ is generically regular and $a$ is generically finite), so this cokernel is generically trivial, i.e. $\omega_A \to a_*\omega$ is generically surjective. In particular, we deduce that $\rank(a_*\omega) = 1$, so $a$ is birational.

	Thus, we are left to show that $\cN = \cN'$. Since the inclusion $\cN' \inc \cN$ is a nil--isomorphism, there exists some $e > 0$ such that the map $\cN \inj V^{es, *}\cN$ factors through $V^{es, *}\cN'$. In particular, \[ \length_0(V^{es, *}\cN') \geq \length_0(\cN) \expl{=}{$\cN$ is suported at $0 \in \bighat{A}$ by assumption.} \length(\cN). \] Given that $V^{es}$ is étale by \autoref{basic facts about AV}.\autoref{itm:ordinary_iff_V_etale}, we see that $\length_0(V^{es, *}\cN') = \length_0(\cN')$, so we have proven that \[ \length(\cN') \geq  \length(\cN). \] Since $\cN' \inc \cN$, this shows that $\cN' = \cN$.
\end{proof}

\begin{rem}
	Under the assumptions of \autoref{V0_zero_dim_implies_done}, we cannot obtain anything better than pure inseparability of $a$. Indeed, any variety whose Albanese morphism is purely inseparable (but non--necessarily birational) satisfies the hypotheses of \emph{loc. cit.}, even when $A$ is ordinary! One then sees from \autoref{V0_zero_dim_implies_done_ordinary_case} that the key additional hypothesis in this case is the assumption $V^0_{\injj} = \{0\}$ (i.e. it should not contain any non--trivial $p$--torsion line bundles).
\end{rem}

As a corollary of \autoref{V0_zero_dim_implies_done}, we can now give a stronger version of \autoref{Pareschi_char_p}:

\begin{thm}\label{more_general_Pareschi}
	Let $X$ be a normal proper variety of maximal Albanese dimension, and assume that its Albanese variety has no supersingular factor. If $\dim(V^0_{\injj}(a_*\omega))) > 0$ for some Cartier module $\omega \inc \omega_X$, then there exists a positive--dimensional closed subset $Z \inc V^0_{\injj}(a_*\omega)$ such that also $-Z \inc V^0_{\injj}(a_*\omega)$.
\end{thm}
\begin{proof}
	As in the proof of \autoref{more_general_Albanese_image}, we may assume that $a \colon X \to A$ is finite. Note that $\omega$ cannot be trivial, so $\tau(\omega_X) \inc \omega$ by definition. If we could show that $\dim(V^0_{\injj}(a_*\tau(\omega_X))) > 0$, then we would be done by \autoref{Pareschi_char_p}. Hence, assume by contradiction that $\dim(V^0_{\injj}(a_*\tau(\omega_X))) = 0$. Then we know by \autoref{V0_zero_dim_implies_done} that $a \colon X \to A$ is purely inseparable. Since $\omega_A = \tau(\omega_A)$ is simple, we deduce by \autoref{universal_homeo_induce_equivalence_of_Cartier_crystals} that the inclusion $\tau(\omega_X) \to \omega_X$ is a nil--isomorphism. In particular, the inclusion $\tau(\omega_X) \inc \omega$ is a nil--isomorphism, forcing the equality $V^0_{\injj}(a_*\tau(\omega_X)) = V^0_{\injj}(a_*\omega)$. This is a contradiction.
\end{proof}

\subsection{Proof of the main theorems}

Here is the main general result from which we will deduce everything else:

\begin{prop}\label{overly_general_result}
	Assume that there exists a Cartier module $\cM$ on $X$ such that:
	\begin{itemize}
		\item $0 \in V^0_{\injj}(a_*\cM)$; 
		\item there exists no positive--dimensional closed subset $Z \inc V^0_{\injj}(a_*\cM)$ such that $-Z \inc V^0_{\injj}(a_*\cM)$.
		\item $A$ has no supersingular factor.
	\end{itemize}
	Then $a$ is surjective with connected fibers. Furthermore, $A$ is ordinary if and only if $H^0_{ss}(X, \cM) \neq 0$.
\end{prop}

\begin{rem}
	Unlike in characteristic zero, we cannot conclude automatically that $a_*\cO_X = \cO_A$ (i.e. that $a$ is a fibration). What we prove here is that the finite part of the Stein factorization of $a$ is purely inseparable. Of course, if $a$ is separable to begin with, then it must be a fibration.
\end{rem}

The first step will be to prove the following:

\begin{lem}\label{first_step_general_proof}
	Assume that there exists a Cartier module $\cM$ on $X$ such that $0 \in V^0_{\injj}(a_*\cM)$ is an isolated point. Then $a$ is surjective and there exists a surjective morphism $a_*\cM \to \omega_A$ of Cartier modules.
	
	Furthermore $H^0_{ss}(X, \cM) \neq 0$ if and only if $A$ is ordinary.
\end{lem}
\begin{proof}
	Let $\cN$ denote the injective $V$--module associated to $a_*\cM$, whence $V^0_{\injj}(a_*\cM) = -\Supp(\cN)$. Since the $V$--module $i_0^{\flat}(\cN)$ is not nilpotent by definition, there exists by \autoref{filtration_for_V_modules_supp_at_0} a non--zero morphism of $V$--crystals $k(0) \dashrightarrow i_0^{\flat}(\cN) \inc \cN$. By \autoref{equivalence_V_crystals_and_Cartier_crystals}, we obtain a non--zero morphism $a_*\cM \dashrightarrow \omega_A$ of Cartier crystals. Since $\omega_A$ is unit, we know by \cite[Corollary 3.4.15]{Baudin_Duality_between_perverse_sheaves_and_Cartier_crystals} that we can represent this morphism of crystals by an actual morphism of Cartier modules. Since $\omega_A$ is a simple Cartier module (\autoref{remark_test_module}.\autoref{rem_test}), this morphism is automatically surjective. In particular, $a_*\cM$ must be supported everywhere, so $a$ is surjective.
	
	Let us show the part after ``Furthermore''. Given that \[ \colim \Tor_0(V^{es, *}\cN, k(0)) \cong \colim \Tor_0(V^{es, *}i_0^{\flat}(\cN), k(0)) \] by assumption, the result follows from \autoref{basic facts about AV_not_so_basic} and the last statement of \autoref{equivalence_V_crystals_and_Cartier_crystals}.
\end{proof}

\begin{proof}[Proof of \autoref{overly_general_result}]
	Since $A$ has no supersingular factor, it follows from the second assumption and \autoref{generic vanishing}.\autoref{itm:gv_p_closed} that $0 \in V^0_{\injj}(a_*\cM)$ is an isolated point. We therefore know by \autoref{first_step_general_proof} that $a$ is surjective and that there is a surjective morphism $a_*\cM \to \omega_A$ of Cartier modules. Let \[ \begin{tikzcd}
		X \arrow[rr, "f"] \arrow[rrrr, "a"', bend right] &  & S \arrow[rr, "g"] &  & A
	\end{tikzcd} \]
	denote the Stein factorization of $a$. By adjunction, the morphism $g_*f_*\cM = a_*\cM \to \omega_A$ gives rise to non--zero morphism $f_*\cM \to g^{\flat}\omega_A \cong \omega_S$ of Cartier modules. Let $\omega$ denote its image, and let $\cK$ denote its kernel. Since $g$ is finite, we obtain a short exact sequence \[ 0 \to g_*\cK \to a_*\cM \to g_*\omega \to 0 \] of Cartier modules on $A$. By \autoref{properties_Fourier_Mukai_transform}.\autoref{itm:support}, applying the Fourier--Mukai transform gives us an exact sequence of $V$--modules \[ \dots \to \cH^0(\FM_A(g_*\omega)) \to \cH^0(\FM_A(a_*\cM)) \to \cH^0(\FM_A(g_*\cK)) \to 0. \] By \autoref{equivalence_V_crystals_and_Cartier_crystals}, the kernel of the first map is a nilpotent $V$--module, so we obtain an injection \[ \cH^0(\FM_A(g_*\omega))_{\injj} \inj \cH^0(\FM_A(a_*\cM)_{\injj} \] (we use the notations from \autoref{canonical_injective_V_module}). This means by definition that $V^0_{\injj}(g_*\omega) \inc V^0_{\injj}(a_*\cM)$. By the second assumption and \autoref{more_general_Pareschi}, we deduce that $V^0_{\injj}(g_*\omega)$ is finite. We therefore conclude by \autoref{V0_zero_dim_implies_done} that $g$ is purely inseparable, so $a$ has connected fibers.
	
	The statement after ``Furthermore'' follows from that of \autoref{first_step_general_proof}.
\end{proof}

We can finally deduce our main theorems:

\begin{thm}\label{main_thm_wo_ss_factor}
	Let $X$ be a normal proper variety of maximal Albanese dimension with $P_2(X) = 1$. Assume that either of the following holds:
	\begin{itemize}
		\item $S^0(X, \omega_X) \neq 0$, or
		\item $\Alb(X)$ has no supersingular factor.
	\end{itemize} 
	Then the Albanese morphism $\alb_X \colon X \to \Alb(X)$ is surjective and purely inseparable. In the first case, $\Alb(X)$ is also ordinary.
	
	If in addition $P_p(X) \leq p - 1$, then $\alb_X$ is birational.
\end{thm}
\begin{proof}
	We want to apply \autoref{overly_general_result} with $\cM = \omega_X$. First, note that the second assumption of \emph{loc. cit.} is satisfied since $P_2(X) = 1$ by hypothesis. Indeed, otherwise there would exist infinitely many $\beta \in \bighat{A}$ such that both $H^0(X, \omega_X \otimes a^*\beta) \neq 0$ and $H^0(X, \omega_X \otimes a^*\beta^{-1}) \neq 0$. Multiplying these sections together and using that $a^*$ is injective (\autoref{existence_of_albanese}) would then give that $P_2(X) > 1$.

	Let us now show that, also in the first case, $\Alb(X)$ has no supersingular factor. By contradiction, suppose that there exists a supersingular elliptic curve $i_E \colon E \inj \bighat{\Alb(X)}$, and let $\cN$ be the $V$--module associated to $\alb_{X, *}(\omega_X)$. Note that $E \not\inc \Supp(\cN)$ by the same argument as in the previous paragraph. Since \[ \colim V^{es, *}\cN \otimes k(0) \neq 0 \] by the last statement of \autoref{equivalence_V_crystals_and_Cartier_crystals}, we obtain that the $V$--module $i_0^{\flat}(i_E^*\cN)$ on $E$ satisfies hypothesis \autoref{basic_facts_AV_itm:there_exists} of \autoref{basic facts about AV_not_so_basic}. By \emph{loc. cit.}, $E$ is actually ordinary, giving a contradiction.  We have shown that $\Alb(X)$ has no supersingular factor, so it is in fact ordinary by \autoref{first_step_general_proof}.
	
	Finally, we know in both cases that $0 \in V^0_{\injj}(\alb_{X, *}(\omega_X))$ by the proof of \cite[Theorem 4.3]{Baudin_Positive_characteristic_generic_vanishing_theory}. The birationality statement follows from \autoref{separability}.
\end{proof}

\begin{thm}\label{main_thm_Delta_integral}
	Let $X$ be a normal proper variety, and let $\Delta \geq 0$ be a $\bZ$--divisor on $X$. Assume that $H^0_{ss}(X, \cO_X(K_X + \Delta)) \neq 0$ and that $P_2(X, \Delta) = 1$. Then $\alb_X$ is surjective with connected fibers and $\Alb(X)$ is ordinary. If in addition $P_p(X, \Delta) \leq p - 1$, then $a$ is a separable fibration.
\end{thm}
\begin{proof}
	To prove the first statement, we want to verify the hypotheses of \autoref{overly_general_result} with $\cM = \cO_X(K_X + \Delta)$. The second and third points hold by the exact same argument as in the proof of \autoref{main_thm_wo_ss_factor}. Since $0 \in W^0_F(\alb_{X, *}(\cO_X(K_X + \Delta)))$ by hypothesis, we deduce by \autoref{generic vanishing}.\autoref{van_and_non_van} that $0 \in V^0_{\injj}(\alb_{X, *}(\cO_X(K_X + \Delta)))$, so we can indeed apply \autoref{overly_general_result} and deduce that $\alb_X$ has connected fibers. If also $P_p(X, \Delta) \leq p - 1$, then we know by \autoref{separability} that $\alb_X$ is separable. In particular, the finite part of the Stein must also be separable. Since it is also purely inseparable, we deduce that it is an isomorphism, i.e. $\alb_{X, *}(\cO_X) = \cO_{\Alb(X)}$.
\end{proof}

\begin{thm}\label{main_thm_kod_dim_zero}
	Let $X$ be a normal proper variety, let $\Delta \geq 0$ be an effective $\bZ_{(p)}$--divisor, and let $\delta  \in \{\pm 1\}$. If $\kappa_S(\delta (K_X + \Delta)) = 0$, then $\Alb(X)$ is ordinary and the Albanese morphism $\alb_X \colon X \to \Alb(X)$ is a fibration. If furthermore $\delta  = 1$, then $\alb_X$ is also separable.
\end{thm}
\begin{proof}
	By \autoref{basic_props_var_kod_zero}, there exists $m > 0$ such that $H^0_{ss}(X, \cO_X(\delta m(K_X + \Delta))) \neq 0$ for some Cartier module structure. We can therefore apply the same proof as in \autoref{main_thm_Delta_integral} to obtain that $\Alb(X)$ is ordinary, that the Albanese morphism of $X$ is surjective, and that the finite part of its Stein factorization is purely inseparable. 
	
	Let us show that it is in fact an isomorphism (we will use the same notations as in the proof of \autoref{overly_general_result}). By \autoref{V0_zero_dim_implies_done_ordinary_case}, we are left to show that $V^0_{\injj}(g_*\omega) = \{0\}$. Let $m > 0$ such that $|\delta m (K_X + \Delta)|$ is non--zero (say it contains the effective divisor $D$). Since $V^0_{\injj}(g_*\omega) \inc V^0_{\injj}(a_*\cO_X(\delta m (K_X + \Delta))) \inc V^0(f_*\cO_X(D))$  (see \autoref{generic vanishing}.\autoref{itm:gv_link_with_section} and the proof of \autoref{overly_general_result}), it is enough to show that there is no non--zero torsion point in $V^0(f_*\cO_X(D))$. By contradiction, assume that there exists $\beta \in \bighat{A}$ non--trivial of $n$--torsion such that $|D + a^*\beta| \neq \emptyset$ (say it contains an effective divisor $E$). Given that $nD \sim nE$ and that $\kappa(X, D) = 0$, we automatically deduce that we have an \emph{equality} of divisors $nD = nE$. This means that $D = E$, or equivalently that $a^*\beta \cong \cO_X$. This contradicts the injectivity of $a^*$ (see \autoref{existence_of_albanese}), so we have proven that $a$ is a fibration.

	If $\delta  = 1$, we can further apply \autoref{separability} to deduce as in the proof of \autoref{main_thm_Delta_integral} that $a$ is a separable fibration.
\end{proof}

Here is a variant of \autoref{main_thm_kod_dim_zero}, generalizes the main result of \cite{Hacon_Patakfalvi_Zhang_Bir_char_of_AVs}.

\begin{cor}\label{main_thm_Hacon_Pat_Zhang}
	Let $X$ be a normal proper variety, let $\Delta \geq 0$ be an effective $\bZ_{(p)}$--divisor, and let $\delta  \in \{\pm 1\}$. If $\kappa(\delta (K_X + \Delta)) = 0$ and $\alb_{X, *}\cO_X(\delta m(K_X + \Delta)) \not\sim_C 0$ for some $m > 0$ and some Cartier module structure on $\cO_X(\delta m(K_X + \Delta))$, then $\alb_X \colon X \to \Alb(X)$ is surjective with connected fibers. If furthermore $\delta  = 1$, then $\alb_X$ is a separable fibration.
\end{cor}

\begin{rem}
	The key difference with \autoref{main_thm_kod_dim_zero} is that in this statement, we do not require $\Alb(X)$ to have no supersingular factor. The non--nilpotence assumption above is for example satisfied in the following situations:
	\begin{itemize}
		\item $X$ has maximal Albanese dimension, $\delta = 1$, $\Delta = 0$ and $m = 1$ (\cite[Theorem 0.2]{Hacon_Patakfalvi_Zhang_Bir_char_of_AVs});
		\item $\Delta$ is a $\bZ$--divisor, $\delta = 1$, $m = 1$ and $H^0_{ss}(X_{\eta}, \omega_{X_{\eta}}(\Delta_{\eta})) \neq 0$, where $X_{\eta}$ denotes the generic fiber to $X \to \alb_X(X)$.
	\end{itemize}
\end{rem}

\begin{proof}
	Let $m > 0$ be as in the statement of \autoref{main_thm_Hacon_Pat_Zhang}, and set $\cM = \cO_X(\delta m (K_X + \Delta))$. The exact same argument as in the proof of \autoref{main_thm_kod_dim_zero} shows that $V^0_{\injj}(\alb_{X, *}\cM)$ does not contain non--zero torsion points. We therefore obtain from \autoref{generic vanishing}.\autoref{itm:gv_p_closed} and \autoref{exists_V^i_with_codim_i} that $V^0_{\injj}(a_*\cM) = \{0\}$. We can then proceed exactly as in the proof of \autoref{overly_general_result} to conclude that $a$ is surjective with connected fibers, and as in the proof of \autoref{main_thm_kod_dim_zero} to obtain that $a$ is a separable fibration when $\delta = 1$.
\end{proof}

Let us finish this section by explaining how to obtain Chen--Hacon's characterization of complex abelian varieties (\cite{Chen_Hacon_Characterization_of_abelian_varieties}) from our result. We will need the following weaker version of the ordinarity conjecture for abelian varieties:

\begin{thm}[\cite{Pink_On_the_order_of_the_reduction_of_a_point_on_an_abelian_variety}]\label{baby_ordinarity_conj_for_AV}
	Let $R$ be a finitely generated and $\bZ$--torsion free $\bZ$--algebra, and let $\cA \to \Spec R$ be an abelian scheme. Then the subset \[ \set{x \in \Spec R \emph{ a closed point}}{\cA_{\overline{x}} \emph{ has no supersingular factor }} \] is dense in $\Spec R$.
\end{thm}
\begin{proof}
	Let $U \inc \Spec R$ be an non--empty open subset. We have to show that for some closed point $x \in U$, $\cA_{\overline{x}}$ has no supersingular factor. Let $\Spec K \to U_{\bQ}$ be a $K$--point of the generic fiber $U_{\bQ}$, where $K$ is a number field (we let $\cO_K$ denote its ring of integers). By usual spreading out, there exists a non--empty open $V \inc \Spec \cO_K$ such that the morphism $\Spec K \to U_{\bQ}$ extends to $f \colon V \to U$. By \cite[Corollary 1.7]{Pink_On_the_order_of_the_reduction_of_a_point_on_an_abelian_variety}, there exists a closed point $y \in V$ such that $\cA_{\overline{y}}$ has no supersingular factor. This is then also the case for $x = f(y)$, so the proof is complete.
\end{proof}

We can now proceed to the characterization of abelian varieties in characteristic zero.

\begin{thm}[\cite{Chen_Hacon_Characterization_of_abelian_varieties}]\label{Chen_Hacon_characterization}
	Let $X$ be a normal proper variety of maximal Albanese dimension over an algebraically closed field of characteristic zero, and assume that $P_2(X) = 1$. Then $X$ is birational to an abelian variety.
\end{thm}
\begin{proof}
	By Chow's lemma, we may assume that $X$ is projective. Let $f \colon \cX \to \Spec S$ be a model of $X$ over a $\bZ$--torsion free and finitely generated $\bZ$--algebra, which is also a domain. Let $U \inc \Spec S$ be a dense open such that $\cX|_U \to U$ is a projective flat fibration with normal geometric fibers (\cite[Proposition 9.9.4]{EGA_IV.3}) and that each $R^if_*\cO_{\cX}$ is locally free. Up to base change and shrinking $U$, we may also assume that $f$ has a section. We therefore know by \cite[9.4.8, 9.5.4 and 9.5.10]{FGA_explained} that $\Pic^0_{\cX_U/U}$ exists and is a smooth projective abelian scheme over $U$. Let $\cA$ denote its dual abelian scheme. Since $f$ has a section, we therefore know by \cite[9.3.11, 9.2.5 and 9.4.3]{FGA_explained} that there exists a Poincaré sheaf on $\Pic^0_{\cX_U/U}$, and we can therefore use it to construct an Albanese map $\cX|_U \to \cA$. Since $X$ has maximal Albanese dimension, we may shrink $U$ and therefore assume that for all closed point $s \in U$, the map $\cX_s \to \cA_s$ is generically finite of fixed degree $d$ (independent of $s$). We may also assume that all residue characteristics of $U$ are higher than $d$, so that each $\cX_s \to \cA_s$ is separable, and finally we may also assume that $P_2(\cX_s) = 1$ for all $s \in U$.
	
	By \autoref{baby_ordinarity_conj_for_AV}, there exists a closed point $s \in U$ such that $\cA_{\overline{s}}$ has no supersingular factor. Since $\cX_s \to \cA_s$ is separable, we obtain by \autoref{main_thm_wo_ss_factor} that $\cX_s \to \cA_s$ is birational, whence $d = 1$. This forces $X \to \Alb(X)$ to be birational too.
\end{proof}

\begin{rem}
	Given the current state of the ordinarity conjecture, we are not able to obtain more general characteristic zero statements with our techniques than this one.
\end{rem}

\begin{example}
	We have proven in particular if $X$ is a normal proper variety such that $\kappa_S(K_X + \Delta) = 0$, then a general fiber of the Albanese morphism is integral since $\alb_X$ is separable. One may wonder if under our assumptions, the general fiber could be for example normal. In the case where $X_{\eta}$ is a curve, or more generally where $K_{X_{\eta}} + \Delta_{\eta} \sim_{\bQ} 0$, then it will follow that $(X_{\eta}, \Delta_\eta)$ is globally sharply $F$--split, so a general fiber will automatically be normal and even $F$--pure. \\
	
	However, this may not be the case in general. We will construct, for any positive characteristic, a smooth threefold $X$ with $\kappa_S(X) = 0$ whose Albanese morphism is a fibration to an elliptic curve such that the general fiber is non--normal, and another one for any $p > 2$, where the general fiber is normal but not $F$--pure. \\
	
	Let $E$ be an ordinary elliptic curve and let $S$ be an ordinary $\mathrm{K3}$ surface. Set $S_{K(E)} \coloneqq S \times_k K(E)$.	Let us first explain the idea behind both counterexamples. Blowing points of $S_{K(E)}$ with purely inseparable residue fields, we find a regular surface $S'$ over $K(E)$ with bad geometric properties (either non geometrically normal, or geometrically normal but non geometrically $F$--pure). Globalizing this, we find some smooth threefold $X$, birational to $E \times S$, whose generic fiber is exactly $S'$. \\
	
	\textbf{Step 1:} Find a birational morphism $Y \to \bA^2_{k(t)}$ (given by a sequence of blowups at regular points), such that $Y$ is regular and either non-geometrically normal, or geometrically normal but not geometrically $F$--pure and $p > 2$. \\
	
	Throughout, we will only blowup regular surfaces at closed points, so all the surfaces we will obtain this way will automatically be regular. Blowing up $\bA^2_{k(t)}$ at the closed point whose ideal is given by $(x^p - t, y)$, we obtain the surface \[ Y_1 \coloneqq \{ (x^p - t)v - yu = 0\} \inc \bA^2_{x, y} \times \bP^1_{u, v}. \] In the chart $u \neq 0$, we obtain \[ \{(x^p - t)v - y = 0\} \inc \bA^3_{x, y, v},\] which is isomorphic to $\bA^2_{x, v}$. In the chart $v \neq 0$, this gives \[ \{ x^p - t - yu = 0\} \inc \bA^3_{x, y, u}. \] Note that by the Jacobian criterion, this is smooth everywhere except at the point $s \in Y_1$ corresponding to the ideal $(x^p - t, y, u)$.  
	
	Let $Y_2$ denote the blowup of $Y_1$ at $s$, and let us look at the equations over the chart $v \neq 0$ ($Y_1$ is smooth on the chart $u \neq 0$). The blowup of $\bA^3_{x, y, u}$ at this point is the closed subscheme of $\bA^3_{x, y, u} \times \bP^2_{a, b, c}$ cut out by the equations 
	\[ \begin{cases}
		(x^p - t)c - ua = 0, \\
		(x^p - t)b - ya = 0, \\
		yc - ub = 0.
	\end{cases} \] 
	Hence, $Y_2$ is given by the following equations (over the chart $v \neq 0$):
	
	\begin{itemize}
		\item in the chart $a \neq 0$, we obtain $\{ (x^p - t)bc - 1 = 0\} \inc \bA^3_{x, b, c}$ (which is smooth);
		\item in the chart $b \neq 0$, we obtain $\{ (x^p - t) - y^2c = 0\} \inc \bA^3_{x, y, c}$ (which is smooth everywhere, except at $(x^p - t, y, c)$);
		\item in the chart $c \neq 0$, we obtain $\{ (x^p - t) - v^2b = 0\} \inc \bA^3_{x, v, b}$ (which is smooth everywhere, except at $(x^p - t, v, b)$).
	\end{itemize}
	
	Note that the non--smooth points in the charts $b \neq 0$ and $c \neq 0$ are the same. Thus, $Y_2$ is smooth everywhere except at one point, and around this point, it is given by $\{x^p - t - y^2c = 0\} \inc \bA^3_{x, y, c}$.
	
	For all $2 \leq i \leq p$, we inductively define $Y_i$ to be the blowup of $Y_{i - 1}$ at its unique non-smooth point. It will then be smooth everywhere, except at one point, which has a neighborhood isomorphic to \[ \{x^p - t - y^ic\} \inc \bA^3_{x, y, c}.\]
	
	When base changing to the algebraic closure $Y_i \times_{k(t)} \overline{k(t)}$, this neighborhood becomes \[ \ttilde{x}^p - y^ic \inc \bA^3_{\ttilde{x}, y, c}, \] where we made the change of variables $\ttilde{x} \coloneqq x - t^{1/p}$. By Fedder's criterion, this is never $F$--pure. Furthermore, this is normal if and only if $i < p$, so we conclude this step. \\
	
	\textbf{Step 2:} There exists a birational morphism $S' \to S_{K(E)}$, such that $S'$ satisfies the same geometric singularity properties as $Y$. \\
	
	Since $S$ is smooth, there is some open $U \inc S_{K(E)}$ and the étale morphism $U \to \bA^2_{K(E)}$. Since $K(E)$ is a separable extension of $k(t)$, we obtain an étale morphism $U \to \bA^2_{k(t)}$. By translating, we may assume that the image of $U$ contains the point we blew up in Step 1. Define \[ Y' \coloneqq U \times_{\bA^2_{k(t)}} Y. \] Since $Y' \to Y$ is étale, $Y'$ has the same geometric singularity properties as $Y$. Furthermore, since $Y' \to U$ is birational, we know by \cite[Theorem II.7.17]{Hartshorne_Algebraic_Geometry} that it is given by a blowup at some closed subscheme $Z \inc U$. Since $U$ is a regular surface, we may assume that $Z$ has codimension at least $2$ (see \cite[Exercise 7.11.(c)]{Hartshorne_Algebraic_Geometry}), and hence is a finite union of points (one could have also computed that directly, without using this result). In particular, $Z$ is again closed in $S_{K(E)}$. If we let $S'$ be the blowup of $S_{K(E)}$ at the same closed subscheme $Z$, then $S'$ satisfies the required properties. \\
	
	\textbf{Step 3:} Globalize the construction and obtain the example. \\
	
	Since $S$ is a $\mathrm{K3}$ surface and $E$ is an elliptic curve, the Albanese morphism of $E \times S$ is given by the projection to $E$. In particular, its generic fiber is exactly $S_{K(E)}$. Define $X'$ as the blowup of $E \times S$ at the closure of the image of $Z \to S_{K(E)} \inc E \times S$. Since $S'$ is regular, $X$ is smooth around the image of $S' \to X'$. 
	
	Let $X \to X'$ be a resolution of singularities leaving invariant the smooth locus of $X'$. Then the generic fiber of the composition $X \to E$ is still $S'$, and $\kappa(K_X) = \kappa(K_{E \times S}) = 0$. Moreover, let $\pi$ denote the composition $X \to X' \to E \times S$. Since $E \times S$ is smooth, we know by \cite[Theorem 1.1]{Rulling_Chatzistamatiou_GR_for_smooth} and duality that \[ \begin{cases*}
		R\pi_*\cO_X \cong \cO_{E \times S}; \\
		\pi_*\omega_X \cong \omega_{E \times S}.
	\end{cases*} \] Given that $H^0_{ss}(E \times S, \omega_{E \times S}) \neq 0$ ($E$ and $S$ are ordinary), we obtain that also $H^0_{ss}(X, \omega_X) \neq 0$, whence $\kappa_S(K_X) = 0$. 

	To conclude, we are then left to show that the Albanese morphism of $X$ is given by the composition $\pi'$ of $X \to E \times S$ and the projection $E \times S \to E$. By \autoref{existence_of_albanese}, we have a factorization $X \to \Alb(X) \to E$ of $\pi'$, so given that $h^1(X, \cO_X) = h^1(E \times S, \cO_{E\times S}) = 1$ and $\Alb(X) \to E$ is surjective, $\Alb(X)$ is automatically a curve. In particular, $\Alb(X) \to E$ is finite. Since $\pi'_*\cO_X = \cO_E$, we deduce that $\Alb(X) = E$.
\end{example}

\section{A bound independent of the characteristic in the case of maximal Albanese dimension}

Let us first recall the following result about Iitaka fibrations. Although this is most likely well--known, we could not find a reference.

\begin{sprop}\label{Iitaka_that_I_need}
	Let $X$ be a normal proper variety with canonical singularities, and assume that $\kappa(X, K_X) \geq 0$. Then if we take the Iitaka fibration of $K_X$: 
	\[ \begin{tikzcd}
		Y \arrow[r, "\pi"] \arrow[d, "g"'] & X \arrow[ld, dashed] \\
		Z                                  &                     
	\end{tikzcd} \] and let $\eta$ denote the generic point of $Z$, then $\kappa(Y_{\eta}, K_{Y_{\eta}}) = 0$.
\end{sprop}
\begin{proof}
	By \cite[Theorem 2.1.33]{Lazarsfeld_Positivity_in_algebraic_geometry_I}, we know that $\kappa(Y_{\eta}, (\pi^*K_X)|_{Y_{\eta}}) = 0$. Since $X$ has canonical singularities, we can write $K_Y \sim_{\bQ} \pi^*K_X + E$ for an effective and $\pi$--exceptional divisor $E$ on $Y$. An issue is that a general fiber of $g$ may not be normal, so let us first reduce to this case. Given $N \gg 0$, consider the commutative square 
	\[ \begin{tikzcd}
		W \arrow[r, "f"] \arrow[d, "h"'] & Y \arrow[d, "g"] \\
		Z \arrow[r, "F^N"']               & Z,               
	\end{tikzcd} \] where we set $W$ to be the normalization of $(Y \times_{Z, F^N} Z)_{\red}$, and $f$ and $h$ are the induced morphisms. We then know by \cite[Lemma 2.4]{Ji_Waldron_Structure_of_geometrically_non_reduced_varieties} that the geometric generic fiber of $h$ is normal, hence so is a general fiber $H$ by \cite[Proposition 9.9.4]{EGA_IV.3}. Since the induced map $W_{\eta} \to Y_{\eta}$ is finite, we know that $\kappa(W_{\eta}, f^*\pi^*K_X|_{W_{\eta}}) = 0$, and we need to show that $\kappa(W_{\eta}, (f^*(\pi^*K_X + E))|_{W_{\eta}}) = 0$ (recall that $K_Y \sim_{\bQ} \pi^*K_X + E$). 

	We therefore have to show that for a given $m > 0$ big and divisible enough and $H$ a general enough fiber (depending on $m$), it holds that $h^0(H, \cO_H(m(f^*(\pi^*K_X + E))|_H)) = 1$. Fix an integer $m > 0$, and take a general enough fiber $H$ so that it is normal, that $h^0(H, \cO_H(mf^*\pi^*K_X|_H)) = 1$, that $f^*E|_H$ is effective and $(f \circ \pi)$--exceptional and that $H \to \pi(f(H))$ is generically finite. 

	Consider the Stein factorization $H \xrightarrow{\mu} T \xrightarrow{\phi} \pi(f(H))$. Since $mf^*E|_H$ is therefore effective and $\mu$--exceptional, we know by \cite[Example 2.1.16]{Lazarsfeld_Positivity_in_algebraic_geometry_I} that \[ h^0(H, \cO_H(mf^*(\pi^*K_X + E)|_H)) = h^0(H, \cO_H(m(f^*\pi^*K_X)|_H) = 1, \] so the proof is complete. \qedhere
	
	%	An issue is that a general fiber of $g$ may not be normal, so let us first reduce to this case. Given $N \gg 0$, consider the commutative square 
	%	\[ \begin{tikzcd}
		%		W \arrow[r, "f"] \arrow[d, "h"'] & Y \arrow[d, "g"] \\
		%		Z \arrow[r, "F^N"']               & Z,               
		%	\end{tikzcd} \] where we set $W$ to be the normalization of $(Y \times_{Z, F^N} Z)_{\red}$. We then know by \JB{ref} that the geometric generic fiber of $h$ is normal, hence so is a general fiber $H$ by \JB{ref EGA 4.3} (and $f$ and $h$ are the induced morphisms). Set $D \coloneqq \pi^*K_X$ and $R \coloneqq f^*E$. Since the induced map $W_{\eta} \to Y_{\eta}$ is finite, we know that $\kappa(W_{\eta}, f^*D|_{W_{\eta}}) = 0$, and we need to show that $\kappa(W_{\eta}, (f^*D + R)|_{W_{\eta}}) = 0$ (recall that $f^*K_Y \sim_{\bQ} f^*D + R$). We therefore have to show that for a given $m > 0$ big and divisible enough and $H$ a general enough fiber (depending on $m$), it holds that $h^0(H, \cO_H(m(f^*D + R)|_H)) = 1$. Take $H$ general enough so that it is normal, that $h^0(H, \cO_H(mf^*D)) = 1$ and that $R|_H$ is $(f \circ \pi)$--exceptional. Consider the Stein factorization $H \xrightarrow{\mu} T \to \pi(f(H))$ (hence $T$ is also normal). Then we can write $m(f^*D + R)|_H$ as $\mu^*B + R'$, where $R'$ is $\mu$--exceptional. \JB{try again without creating 10000 notations!!}
\end{proof}

We will adapt the techniques from \cite{Zhang_Pluricanonical_maps_of_varieties_of_maximal_albanese_dimension_in_positive_characteristic} in order to obtain the following:

\begin{sthm}\label{main_thm_Zhang_technique}
	Let $X$ be a normal proper variety of maximal Albanese dimension with canonical singularities. Assume further that $\kappa_S(K_X) \geq 0$ and that $P_4(X) = 1$. Then $X \to \Alb(X)$ is birational.
\end{sthm}
\begin{srem}
	\begin{itemize}
		\item If $\Alb(X)$ is ordinary, then it follows from \cite[Theorem 4.3]{Baudin_Positive_characteristic_generic_vanishing_theory} that $S^0(X, \omega_X) \neq 0$. In particular, $\kappa_S(K_X) \geq 0$.
		\item Note that any variety of general type $X$ satisfies $\kappa_S(K_X) \geq 0$ (see \cite[Lemma 4.1.5]{Hacon_Pat_GV_Characterization_Ordinary_AV}), so \autoref{main_thm_Zhang_technique} shows that any variety of general type and maximal Albanese dimension satisfies $P_4(X) \geq 2$.
		\item The proof will not tell us anything about $V^0_{\injj}(a_*\omega_X)$. In particular, the technique of the proof of \autoref{overly_general_result} does not allow us to obtain results about varieties which do not have maximal Albanese dimension.
	\end{itemize}
\end{srem}
\begin{proof}
	We may assume that the base field $k$ is uncountable. We will use the exact same notations as in the proof of \autoref{Iitaka_that_I_need}. If $\kappa(K_X) = 0$, then we know by \autoref{main_thm_Hacon_Pat_Zhang} or \cite[Theorem 0.1]{Hacon_Patakfalvi_Zhang_Bir_char_of_AVs} that $X \to \Alb(X)$ is birational, so assume that $\kappa(K_X) > 0$. 
	
	For $e \gg 0$, consider the trace map \[ F^e_*\cO_Y((1 + p^e)K_Y) \to \cO_Y(2K_Y). \] Let us show that its pushforward under $g$ does not vanish. This is equivalent to showing that $S^0(Y_{\eta}, \cO_{Y_{\eta}}(2K_{Y_{\eta}})) \neq 0$, where $\eta$ denotes the generic point of $Z$. Note that since $\kappa_S(K_X) \geq 0$, we know that also $\kappa_S(K_Y) \geq 0$, since $\pi_*\cO_Y(mK_Y) \cong \cO_X(mK_X)$ for all $m \geq 0$ ($X$ has canonical singularities). In particular, $\kappa_S(K_{Y_{\eta}}) = 0$. Given that $K_Y \geq 0$ by \cite[Theorem 4.3]{Baudin_Positive_characteristic_generic_vanishing_theory}, we also have that $K_{Y_{\eta}} \geq 0$, so $S^0(Y_{\eta}, \cO_{Y_{\eta}}(2K_{Y_{\eta}})) \neq 0$ by \autoref{basic_props_var_kod_zero}.
	
	Let us show that the image of a very general fiber $G$ of $g$ to $\Alb(X)$ is the translate of a fixed abelian subvariety. By \cite[Theorem 1.1]{Patakfalvi_Waldron_Singularities_of_general_fibers}, we know that $K_{W_{\eta}} \leq f^*K_{Y_{\eta}}$. Since $K_W \geq 0$ by \cite[Theorem 4.3]{Baudin_Positive_characteristic_generic_vanishing_theory}, we deduce that \autoref{Iitaka_that_I_need} that $\kappa(K_{W_{\eta}}) = 0$. In particular, we also have that $\kappa(K_H) = 0$ for a very general fiber $H$ of $h$. Consider the composition $H \to \Alb(H) \to \Alb(X)$. Since $H \to \Alb(H)$ is in particular surjective by \autoref{main_thm_Hacon_Pat_Zhang}, we deduce that the image of $H \to \Alb(X)$ is a translate of an abelian subvariety $K^H \inc \Alb(X)$. Since an abelian variety has only countably many abelian subvarieties, we deduce that $K^H = K$ does not depend on $H$. 
	
	Consider $B \coloneqq \Alb(X)/K$, so that $H$ gets contracted by the composition $H \to \Alb(X) \to B$. This is therefore also the case of a very general fiber $G = g(H)$ of $g$, so by the rigidity lemma we obtain a rational map $b$ fitting in the commutative diagram 
	\[ \begin{tikzcd}
		Y \arrow[r, "\pi"] \arrow[d, "g"'] & X \arrow[r, "\alb_X"] & \Alb(X) \arrow[d] \\
		Z \arrow[rr, "b", dashed]          &                       & B.               
	\end{tikzcd} \] Up to modifying $Z$ birationally and $Y$ accordingly, we may assume that $b$ is defined everywhere. By construction of the Iitaka fibration (and by our rational modification), there exists an big and semiample Cartier divisor $D'$ on $Z$ such that $g^*D' \leq \pi^*(nK_X) \leq nK_Y$ for some $n > 0$. By bigness, some power of $D'$ therefore contains an ample divisor $D$ and since $K_Y \geq 0$, we can consider the composition \[ F^e_*\cO_Y(mg^*D) \inc F^e_*\cO_Y((1 + p^e)K_Y) \to \cO_Y(2K_Y) \] with $m \gg 0$ (recall that $e \gg 0$). We can therefore push it forward to $Z$ and by our work above, we obtain a non--zero morphism \[ F^e_*\cO_Z(mD) \to g_*\cO_Y(2K_Y). \] Since $m \gg 0$, we have by Fujita vanishing (\cite{Fujita_Vanishing_theorem_for_semipositive_line_bundles}) that for all $\beta \in \bighat{B}$ and $i > 0$, $H^i(Z, F^e_*\cO_Z(mD) \otimes b^*\beta) = 0$ and $R^ib_*(F^e_*\cO_Z(mD) \otimes b^*\beta) = 0$. We therefore obtain by \cite[Proposition 2.8]{Pareschi_Popa_Regularity_on_AV_III} that $\cV \coloneqq \FM_B(b_*(F^e_*\cO_Z(mD)))$ is in particular a torsion--free sheaf in degree zero (even a vector bundle actually).
	
	Since we have a non--zero morphism $b_*F^e_*\cO_Z(mD) \to b_*g_*\cO_Y(2K_Y)$, we obtain by \autoref{properties_Fourier_Mukai_transform}.\autoref{itm:equiv_cat} a non--zero morphism \[ \FM_B(b_*g_*\cO_Y(2K_Y)) \to \cV. \] By \autoref{properties_Fourier_Mukai_transform}.\autoref{itm:support}, we therefore obtain a non--zero morphism \[ \cH^0(\FM_B(b_*g_*(2K_Y))) \to \cV. \] Since $\cV$ is torsion--free, the image of this morphism is also torsion--free. In particular, $\Supp(\cH^0(\FM_B(b_*g_*\cO_Y(2K_Y)))) = \bighat{B}$, so \autoref{properties_Fourier_Mukai_transform}.\autoref{itm:support_H0} gives us that for all $\beta \in \bighat{B}$, \[ H^0(X, \cO_X(2K_X) \otimes \alb_X^*\beta) \neq 0. \] Since $b$ is generically finite and $\dim(Z) = \kappa(K_X) > 0$, we have that $\dim(\bighat{B}) > 0$. The same argument as in the beginning of the proof of \autoref{main_thm_wo_ss_factor} then shows that $P_4(X) > 1$, contradicting our hypotheses.
\end{proof}

\bibliographystyle{alpha}
\bibliography{Bibliography}

\newcommand{\etalchar}[1]{$^{#1}$}
\begin{thebibliography}{FGI{\etalchar{+}}05}

\bibitem[Bau23]{Baudin_Duality_between_perverse_sheaves_and_Cartier_crystals}
J.~Baudin.
\newblock Duality between {C}artier crystals and perverse
  $\mathbb{F}_p$-sheaves, and application to generic vanishing.
\newblock {\em arXiv e-print: arXiv:2306.05378v2}, 2023.
\newblock Available at
  \href{https://arxiv.org/abs/2306.05378}{arXiv:2306.05378}.

\bibitem[Bau25]{Baudin_Positive_characteristic_generic_vanishing_theory}
J.~Baudin.
\newblock Generic vanishing theory in positive characteristic.
\newblock {\em arXiv e-print: arXiv:2507.00771v1}, 2025.
\newblock Available at
  \href{https://arxiv.org/abs/2507.00771}{arXiv:2507.00771}.

\bibitem[BB11]{Blickle_Bockle_Cartier_modules_finiteness_results}
M.~Blickle and G.~B\"{o}ckle.
\newblock Cartier modules: finiteness results.
\newblock {\em J. Reine Angew. Math.}, 661:85--123, 2011.

\bibitem[BBK23]{Baudin_Bernasconi_Kawakami_Frobenius_GR_fails}
J.~Baudin, F.~Bernasconi, and T.~Kawakami.
\newblock The {F}robenius--stable version of the {G}rauert--{R}iemenschneider
  vanishing theorem fails.
\newblock {\em arXiv e-print: arXiv:2312.13456v3}, 2023.
\newblock Available at
  \href{https://arxiv.org/abs/2312.13456}{arXiv:2312.13456}.

\bibitem[BP09]{Bockle_Pink_Cohomological_Theory_of_crystals_over_function_fields}
G.~B\"{o}ckle and R.~Pink.
\newblock {\em Cohomological theory of crystals over function fields}, volume~9
  of {\em EMS Tracts in Mathematics}.
\newblock European Mathematical Society (EMS), Z\"{u}rich, 2009.

\bibitem[CH01]{Chen_Hacon_Characterization_of_abelian_varieties}
J.~A. Chen and C.~D. Hacon.
\newblock Characterization of abelian varieties.
\newblock {\em Invent. Math.}, 143(2):435--447, 2001.

\bibitem[CJ18]{Chen_Jiang_Positivity_in_varieties_of_maximal_Albanese_dimension}
J.~A. Chen and Z.~Jiang.
\newblock Positivity in varieties of maximal {A}lbanese dimension.
\newblock {\em J. Reine Angew. Math.}, 736:225--253, 2018.

\bibitem[CR15]{Rulling_Chatzistamatiou_GR_for_smooth}
A.~Chatzistamatiou and K.~R\"{u}lling.
\newblock Vanishing of the higher direct images of the structure sheaf.
\newblock {\em Compos. Math.}, 151(11):2131--2144, 2015.

\bibitem[Fer19]{Ferrari_An_Enriques_classification_theorem_for_surfaces_in_positive_characteristic}
E.~Ferrari.
\newblock An {E}nriques classification theorem for surfaces in positive
  characteristic.
\newblock {\em Manuscripta Math.}, 160(1-2):173--185, 2019.

\bibitem[FGI{\etalchar{+}}05]{FGA_explained}
B.~Fantechi, L.~G\"{o}ttsche, L.~Illusie, S.~L. Kleiman, N.~Nitsure, and
  A.~Vistoli.
\newblock {\em Fundamental algebraic geometry}, volume 123 of {\em Mathematical
  Surveys and Monographs}.
\newblock American Mathematical Society, Providence, RI, 2005.
\newblock Grothendieck's FGA explained.

\bibitem[Fil18]{Filipazzi_GV_fails_in_pos_char}
S.~Filipazzi.
\newblock Generic vanishing fails for surfaces in positive characteristic.
\newblock {\em Boll. Unione Mat. Ital.}, 11(2):179--189, 2018.

\bibitem[Fuj83]{Fujita_Vanishing_theorem_for_semipositive_line_bundles}
T.~Fujita.
\newblock Vanishing theorems for semipositive line bundles.
\newblock In {\em Algebraic geometry ({T}okyo/{K}yoto, 1982)}, volume 1016 of
  {\em Lecture Notes in Math.}, pages 519--528. Springer, Berlin, 1983.

\bibitem[Gab04]{Gabber_notes_on_some_t_structures}
O.~Gabber.
\newblock Notes on some {$t$}-structures.
\newblock In {\em Geometric aspects of {D}work theory. {V}ol. {I}, {II}}, pages
  711--734. Walter de Gruyter, Berlin, 2004.

\bibitem[GL87]{Green_Lazarsfeld_Generic_vanishing}
M.~Green and R.~Lazarsfeld.
\newblock Deformation theory, generic vanishing theorems, and some conjectures
  of {E}nriques, {C}atanese and {B}eauville.
\newblock {\em Invent. Math.}, 90(2):389--407, 1987.

\bibitem[GL91]{Green_Lazarsfeld_GV_2}
M.~Green and R.~Lazarsfeld.
\newblock Higher obstructions to deforming cohomology groups of line bundles.
\newblock {\em J. Amer. Math. Soc.}, 4(1):87--103, 1991.

\bibitem[Gro66]{EGA_IV.3}
A.~Grothendieck.
\newblock \'{E}l\'{e}ments de g\'{e}om\'{e}trie alg\'{e}brique. {IV}. \'{E}tude
  locale des sch\'{e}mas et des morphismes de sch\'{e}mas. {III}.
\newblock {\em Inst. Hautes \'{E}tudes Sci. Publ. Math.}, (28):255, 1966.

\bibitem[Har77]{Hartshorne_Algebraic_Geometry}
R.~Hartshorne.
\newblock {\em Algebraic geometry}.
\newblock Springer-Verlag, New York-Heidelberg, 1977.
\newblock Graduate Texts in Mathematics, No. 52.

\bibitem[HK15]{Hacon_Kovacs_GV_fails_in_pos_char}
C.~D. Hacon and S.~J. Kov\'{a}cs.
\newblock Generic vanishing fails for singular varieties and in characteristic
  {$p>0$}.
\newblock In {\em Recent advances in algebraic geometry}, volume 417 of {\em
  London Math. Soc. Lecture Note Ser.}, pages 240--253. Cambridge Univ. Press,
  Cambridge, 2015.

\bibitem[HP16]{Hacon_Pat_GV_Characterization_Ordinary_AV}
C.~D. Hacon and Zs. Patakfalvi.
\newblock Generic vanishing in characteristic {$p>0$} and the characterization
  of ordinary abelian varieties.
\newblock {\em Amer. J. Math.}, 138(4):963--998, 2016.

\bibitem[HP22]{Hacon_Pat_GV_Geom_Theta_Divs}
C.~D. Hacon and Zs. Patakfalvi.
\newblock Generic vanishing in characteristic {$p>0$} and the geometry of theta
  divisors.
\newblock {\em Boll. Unione Mat. Ital.}, 15(1-2):215--244, 2022.

\bibitem[HPS18]{Hacon_Popa_Schnell_Alg_fiber_spaces_over_abelian_varieties}
C.~D. Hacon, M.~Popa, and C.~Schnell.
\newblock Algebraic fiber spaces over abelian varieties: around a recent
  theorem by {C}ao and {P}\u{a}un.
\newblock In {\em Local and global methods in algebraic geometry}, volume 712
  of {\em Contemp. Math.}, pages 143--195. Amer. Math. Soc., [Providence], RI,
  [2018] \copyright 2018.

\bibitem[HPZ19]{Hacon_Patakfalvi_Zhang_Bir_char_of_AVs}
C.~D. Hacon, Zs. Patakfalvi, and L.~Zhang.
\newblock Birational characterization of {A}belian varieties and ordinary
  {A}belian varieties in characteristic {$p>0$}.
\newblock {\em Duke Math. J.}, 168(9):1723--1736, 2019.

\bibitem[Jia11]{Jiang_An_effective_version_of_a_thm_of_Kawamata_on_the_Albanese_map}
Z.~Jiang.
\newblock An effective version of a theorem of {K}awamata on the {A}lbanese
  map.
\newblock {\em Commun. Contemp. Math.}, 13(3):509--532, 2011.

\bibitem[JW21]{Ji_Waldron_Structure_of_geometrically_non_reduced_varieties}
L.~Ji and J.~Waldron.
\newblock Structure of geometrically non-reduced varieties.
\newblock {\em Trans. Amer. Math. Soc.}, 374(12):8333--8363, 2021.

\bibitem[Kaw81]{Kawamata_Characterization_of_abelian_varieties}
Y.~Kawamata.
\newblock Characterization of abelian varieties.
\newblock {\em Compositio Math.}, 43(2):253--276, 1981.

\bibitem[Kle05]{Kleiman_The_Picard_Scheme}
S.~Kleiman.
\newblock The {P}icard scheme.
\newblock In {\em Fundamental algebraic geometry}, volume 123 of {\em Math.
  Surveys Monogr.}, pages 235--321. Amer. Math. Soc., Providence, RI, 2005.

\bibitem[Kol86]{Kollar_Higher_Direct_Image_of_Dualizing_Sheaves_II}
J.~Koll\'{a}r.
\newblock Higher direct images of dualizing sheaves. {II}.
\newblock {\em Ann. of Math. (2)}, 123(1):11--42, 1986.

\bibitem[Laz04]{Lazarsfeld_Positivity_in_algebraic_geometry_I}
R.~Lazarsfeld.
\newblock {\em Positivity in algebraic geometry. {I}}, volume~48 of {\em
  Ergebnisse der Mathematik und ihrer Grenzgebiete. 3. Folge. A Series of
  Modern Surveys in Mathematics [Results in Mathematics and Related Areas. 3rd
  Series. A Series of Modern Surveys in Mathematics]}.
\newblock Springer-Verlag, Berlin, 2004.
\newblock Classical setting: line bundles and linear series.

\bibitem[LS21]{Schroer_Laurent_Para_Abelian_varieties_and_Albanese_maps}
B.~Laurent and S.~Schr\"{o}er.
\newblock Para-abelian varieties and albanese maps.
\newblock {\em arXiv e-print: arXiv:2101.10829v2}, 2021.
\newblock Available at
  \href{https://arxiv.org/abs/2101.10829}{arXiv:2101.10829}.

\bibitem[Muk81]{Mukai_Fourier_Mukai_transform}
S.~Mukai.
\newblock Duality between {$D(X)$} and {$D(\hat X)$} with its application to
  {P}icard sheaves.
\newblock {\em Nagoya Math. J.}, 81:153--175, 1981.

\bibitem[Mum08]{Mumford_Abelian_Varieties}
D.~Mumford.
\newblock {\em Abelian varieties}, volume~5 of {\em Tata Institute of
  Fundamental Research Studies in Mathematics}.
\newblock 2008.
\newblock With appendices by C. P. Ramanujam and Y. Manin, Corrected reprint of
  the second (1974) edition.

\bibitem[Par12]{Pareschi_Basic_results_on_irr_vars_via_FM_methods}
G.~Pareschi.
\newblock Basic results on irregular varieties via {F}ourier-{M}ukai methods.
\newblock In {\em Current developments in algebraic geometry}, volume~59 of
  {\em Math. Sci. Res. Inst. Publ.}, pages 379--403. Cambridge Univ. Press,
  Cambridge, 2012.

\bibitem[Pin04]{Pink_On_the_order_of_the_reduction_of_a_point_on_an_abelian_variety}
R.~Pink.
\newblock On the order of the reduction of a point on an abelian variety.
\newblock {\em Math. Ann.}, 330(2):275--291, 2004.

\bibitem[PP11]{Pareschi_Popa_Regularity_on_AV_III}
G.~Pareschi and M.~Popa.
\newblock Regularity on abelian varieties {III}: relationship with generic
  vanishing and applications.
\newblock In {\em Grassmannians, moduli spaces and vector bundles}, volume~14
  of {\em Clay Math. Proc.}, pages 141--167. Amer. Math. Soc., Providence, RI,
  2011.

\bibitem[PR04]{Pink_Roessler_Manin_Mumford_conjecture}
R.~Pink and D.~Roessler.
\newblock On {$\psi$}-invariant subvarieties of semiabelian varieties and the
  {M}anin-{M}umford conjecture.
\newblock {\em J. Algebraic Geom.}, 13(4):771--798, 2004.

\bibitem[PW22]{Patakfalvi_Waldron_Singularities_of_general_fibers}
Zs. Patakfalvi and J.~Waldron.
\newblock Singularities of general fibers and the {LMMP}.
\newblock {\em Amer. J. Math.}, 144(2):505--540, 2022.

\bibitem[PZ21]{Patakfalvi_Zdanowicz_Ordinary_varieties_with_trivial_canonical_bundle_are_not_uniruled}
Zs. Patakfalvi and M.~Zdanowicz.
\newblock Ordinary varieties with trivial canonical bundle are not uniruled.
\newblock {\em Math. Ann.}, 380(3-4):1767--1799, 2021.

\bibitem[Sch22]{Schnell_Fourier_Mukai_transform_made_easy}
C.~Schnell.
\newblock The {F}ourier-{M}ukai transform made easy.
\newblock {\em Pure Appl. Math. Q.}, 18(4):1749--1770, 2022.

\bibitem[Smi95]{Smith_Test_ideals_in_local_rings}
K.~E. Smith.
\newblock Test ideals in local rings.
\newblock {\em Trans. Amer. Math. Soc.}, 347(9):3453--3472, 1995.

\bibitem[SS10]{Schwede_Smith_Globally_F_regular_and_log_Fano_varieties}
K.~Schwede and K.~E. Smith.
\newblock Globally {$F$}-regular and log {F}ano varieties.
\newblock {\em Adv. Math.}, 224(3):863--894, 2010.

\bibitem[{Sta}25]{Stacks_Project}
The {Stacks project authors}.
\newblock The {S}tacks project.
\newblock \url{https://stacks.math.columbia.edu}, 2025.

\bibitem[Zha14]{Zhang_Pluricanonical_maps_of_varieties_of_maximal_albanese_dimension_in_positive_characteristic}
Y.~Zhang.
\newblock Pluri-canonical maps of varieties of maximal {A}lbanese dimension in
  positive characteristic.
\newblock {\em J. Algebra}, 409:11--25, 2014.

\bibitem[Zha19]{Zhang_Abundance_for_non_uniruled_3_folds_with_non_trivial_albanese_in_pos_char}
L.~Zhang.
\newblock Abundance for non-uniruled 3-folds with non-trivial {A}lbanese maps
  in positive characteristics.
\newblock {\em J. Lond. Math. Soc. (2)}, 99(2):332--348, 2019.

\end{thebibliography}

\Addresses

\end{document}